\documentclass[sn-basic]{sn-jnl}

\jyear{2022}%

\theoremstyle{thmstyleone}%
\newtheorem{theorem}{Theorem}
\newtheorem{lemma}{Lemma}
\newtheorem{assumption}{Assumption}

\theoremstyle{thmstyletwo}%
\newtheorem{remark}{Remark}%

\theoremstyle{thmstylethree}%
\begin{document}

\title[Optimal subsampling for functional  quantile regression]{Optimal subsampling for functional quantile regression}


\author[1]{\fnm{Qian} \sur{Yan}}\email{qianyan@cqu.edu.cn}

\author*[1]{\fnm{Hanyu} \sur{Li}}\email{lihy.hy@gmail.com or hyli@cqu.edu.cn}

\author[1]{\fnm{Chengmei} \sur{Niu}}\email{chengmeiniu@cqu.edu.cn}

\affil[1]{\orgdiv{College of Mathematics and Statistics}, \orgname{Chongqing University}, \orgaddress{
		\city{Chongqing }, \postcode{401331},  \country{ P.R. China}}}


\abstract{Subsampling 
	is an efficient method to deal with massive data. 
	In this paper, we investigate the optimal subsampling for linear quantile regression when the covariates are functions. The asymptotic distribution of the subsampling estimator is first derived. Then, we obtain the optimal subsampling probabilities based on the A-optimality criterion. Furthermore, the modified subsampling probabilities without estimating the densities of the response variables given the covariates are also proposed, which are easier to implement in practise. Numerical experiments on synthetic and real data show that the proposed methods always outperform the one with uniform sampling and can approximate the results based on full data well with  less computational efforts. 
}



\maketitle

\section{Introduction}\label{sec.1}
Technological advances have made data easier to collect, store, and process, allowing multiple points in the temporal or spatial domain to be observed and recorded. These observations can be viewed as smooth functions with respect to time or space, which is called functional data in statistics. Functional data analysis is particularly important given the widespread availability of functional data. Traditional statistical methods, however, are no longer available due to limited computer resources as a result of these massive data. In order to overcome this problem, random subsampling methods are alternative approaches that have shown good performance in extracting meaningful information from big datasets and making statistical methods scalable to massive data.

To the best of our knowledge, there are two main types of random subsampling methods in statistical models: Randomized Numerical Linear Algebra (RandNLA) subsampling approaches and optimal subsampling approaches. Popular RandNLA subsampling approaches include uniform sampling, leverage score sampling and shrinkage leverage score sampling; see e.g., 
\citep{Drineas2006, Mahoney2011, Drineas2012}. 
Currently, some researchers have studied statistical properties of these RandNLA subsampling estimators for regression models. For example, 
\cite{Ma2015property} presented the bias and variance of subsampling estimator for least squares regression, and 
\cite{Wang2018property} and 
\cite{Homrighausen2019property} extended them to ridge regression. 
\cite{Raskutti2016property} and \cite{Dobriban2019property} investigated error bounds for the statistical efficiency on estimator based on subsampling least squares regression. 

On the other hand, several scholars have developed optimal subsampling methods for parametric regression problems. For example, 
\cite{Wang2018Optimal} proposed an inverse weighted subsampling method for logistic regression based on the A- or L-optimality criterion. Subsequently, a more efficient estimation method and Poisson subsampling were considered by \cite{Wang2019efficient} 
to correct the bias of the subsampling estimator given in \cite{Wang2018Optimal} and to improve the computational efficiency. Later, 
\cite{Yao2019softmax} and \cite{Ai2021optimal} extended the subsampling method to softmax regression and generalized linear models, respectively. Very recently, 
\cite{Wang2021quantile}, \cite{Ai2021quantile}, \cite{Fan2021quantile}, and \cite{Zhou2022quantile} employed the optimal subsampling method to ordinary quantile regression, and 
\cite{Shao2021quantile} and \cite{Yuan2022quantile} developed the subsampling for composite quantile regression. 

All of the aforementioned studies of subsampling methods focus on statistical models with scalar variables, and now only little work has been done in the area of subsampling for 
functional regression. As far as we know, 
these studies are mainly concerned with functional mean regression, which is an extension of the multiple mean regression model in the functional data setting. Specifically, 
\cite{he2022functional} proposed a functional principal subspace sampling probability for functional linear regression with scalar response, which eliminates the impact of eigenvalue inside the functional principal subspace and properly weights the residuals. 
\cite{liu2021functional} extended the optimal subsampling method to 
functional linear regression and functional generalized linear model with a scalar response. As we know, 
the quantile regression model proposed by 
\cite{Koenker1978quantile}  gives much more complete information about the conditional response distribution than the 
traditional mean regression, and exhibits robustness to outliers and data located in the tail of the conditional response distribution. {For functional quantile regression with scalar response, there are also 
many works; see e.g., \citep{ Cardot2004quantile, Cardot2005quantile, Chen2012, Kato2012, Sang2020quantile}.  
More specifically, \cite{ Cardot2004quantile, Cardot2005quantile} studied penalized spline estimator and its convergence rate. 
\cite{Chen2012} and \cite{Kato2012}  obtained the estimation of slope function based on functional principal component analysis basis. 
\cite{Sang2020quantile} studied penalized spline estimator for functional single index quantile regression.} However, these methods cannot be directly applied to large datasets, and, to the best of our knowledge, there is almost no work on random subsampling for functional quantile regression, in contrast to quantile regression with scalar variables, where there is a lot of work as previously mentioned.

Based on the above motivation, we investigate the 
optimal subsampling for quantile regression in massive data when the covariates are functions. 
We first derive the asymptotic distribution of the general subsampling estimator and then obtain the optimal subsampling probabilities by minimizing the asymptotic integrated mean squared error (IMSE) under the A-optimality criterion. 
In addition, we also provide a feasible modified version of the optimal subsampling probabilities to ensure the feasibility of the subsampling method. These subsampling probabilities are non-informative, which is consistent with the conclusion in 
\cite{Ai2021quantile}. 

The rest of this paper is organized as follows. Section \ref{sec.2} briefly introduces the scalar-on-function linear quantile regression problem and presents asymptotic behaviors of the penalized spline estimator. In Section \ref{sec.3}, we derive 
the asymptotic distribution of the subsampling estimator 
and the optimal subsampling probabilities 
based on the A-optimality criterion. The modified version of these 
probabilities is also considered in this section. Section \ref{sec.4} illustrates our methodology through both numerical simulations and real data sets. Section \ref{sec.5} concludes this paper with some discussions. All proofs are delivered to the Appendix.

\section{ Model and Estimation}\label{sec.2}
\subsection{Functional quantile regression}\label{sec.2.1}

Suppose that $\{x_i(t),y_i\}^n_{i=1} $ are $ n $ independent observations of $ (\boldsymbol{X}(t),\boldsymbol{Y}) $, where the covariates $ x_i(t) $ are square integrable functions defined on $ [0, 1] $, i.e., the elements of the space $ L^2[0, 1] $, and are assumed 
to be non-random, and $ y_i $ are scalar responses. A scalar-on-function linear quantile regression model is defined as follows
\begin{eqnarray}\label{2.1}
	y_i=\int^1_0 x_i(t)\beta(t)\mathrm{d}t+\epsilon_i \quad with \quad \mathrm{P}(\epsilon_i < 0 \mid x_i(t)) = \tau,
\end{eqnarray}
where $ \beta(t) $ is an unknown slope function satisfying $\beta(t)\in L^2[0, 1]$, $\epsilon_i$ are independent random error with probability density function $ f_{\epsilon\mid\boldsymbol{X}(t)}(\epsilon_i,x_i(t)) $, and the quantile level $ \tau \in(0, 1) $. Thus, the $ \tau $-th conditional quantile of $ y_i $ given $ x_i(t) $ is 
\begin{eqnarray*}\label{2.2}
	Q_\tau(y_i\mid x_i(t))=\int^1_0 x_i(t)\beta(t)\mathrm{d}t.		
\end{eqnarray*}

\subsection{Full data estimation of $ \beta(t)$}\label{sec.2.2}
To estimate the slope function $ \beta(t) $, we consider the B-spline basis functions defined on equispaced knots. Specifically, let $ K $ equispaced interior knots divide the interval $ [0,1] $ into $ K+1 $ sub-intervals, i.e., $ [t_j,t_{j+1}], j=0,\dots,K.$ 
In these intervals, we can find $ K + p+1 $ normalized B-spline basis functions $\{B_k(t),1\le k \le K+p+1\} $, as denoted by $ \boldsymbol{B}(t)=(B_1(t),B_2(t),\dots,B_{K+p+1}(t))^T $. They are the piecewise polynomials of degree $ p $  on each sub-interval $ [t_j,t_{j+1}]$ and $ p-1 $ times continuously differentiable on $ [0, 1] $. 
More properties of the B-spline function can be found in 
\cite{Boor2001spline}.
Thus, we can estimate $ \beta(t) $ using a linear combination of the normalized B-spline basis functions \citep{stone1985}, which allows us to find a vector $\boldsymbol{\hat{\theta}} \in \mathbb{R}^{K+p+1} $ such that 
\begin{eqnarray*}
	\hat{\beta}(t)=\sum_{k=1}^{K+p+1}\boldsymbol{\hat{\theta}}_k B_k(t)=\boldsymbol{B}^T(t)\boldsymbol{\hat{\theta}},		
\end{eqnarray*}
where $\boldsymbol{\hat{\theta}} $ is a solution of the minimization problem
\begin{eqnarray}\label{2.3}
	L(\boldsymbol{\theta};\lambda,K)=\sum_{i=1}^{n}\rho_\tau(y_i-\int ^1_0 x_i(t)\boldsymbol{B}^T(t)\boldsymbol{\theta} \mathrm{d}t)+\frac{\lambda}{2}\int^1_0\left\{\left(\boldsymbol{B}^{(q)}(t)\right)^T\boldsymbol{\theta}\right\}^2\mathrm{d}t,		
\end{eqnarray}
where $ \rho_\tau(\epsilon)=\epsilon\{\tau-I(\epsilon<0)\}$ is the quantile loss function with $ I(\cdot) $ being the indicator function, $ \lambda>0  $ is the smoothing parameter, and $\boldsymbol{B}^{(q)}(t) $ in the penalty term is the integrated squared $ q $-th order derivative of all the B-splines functions for some integer $ q \le p $.
Furthermore, let $\boldsymbol{B}_i=\int_{0}^{1}x_i(t)\boldsymbol{B}(t)\mathrm{d}t$ and $ \boldsymbol{D}_q=\int_{0}^{1}\boldsymbol{B}^{(q)}(t)\{\boldsymbol{B}^{(q)}(t)\}^T\mathrm{d}t $, the loss function in (\ref{2.3}) thus can be rewritten as 
\begin{eqnarray}\label{2.4}
	L(\boldsymbol{\theta};\lambda,K)=\sum_{i=1}^{n}\rho_\tau(y_i-\boldsymbol{B}^T_i\boldsymbol{\theta} )+\frac{\lambda}{2}\boldsymbol{\theta}^T\boldsymbol{D}_q\boldsymbol{\theta}.		
\end{eqnarray}

\subsection{Asymptotic theory of $ \hat\beta(t) $}\label{sec.2.3}
In this section, we show the asymptotic properties of $ \hat{\beta}(t)$ based on full data. To get the desired results, here we assume that the following assumptions are satisfied. 
\begin{assumption}\label{A1}
	For the functional covariates $ \boldsymbol{X}(t) $, assume there exist a constant $C_1$ such that $\Vert \boldsymbol{X}(t)\Vert_2 \le C_1 < \infty $ a.s..
\end{assumption}
\begin{assumption}\label{A2}
	Assume the unknown functional coefficient $ \beta(t) $ is sufficiently smooth. That is, $ \beta(t) $ 
	has a $ d'$-th derivative $ \beta^{(d')}(t) $ such that
	\begin{eqnarray*}
		\mid\beta^{(d')}(t)-\beta^{(d')}(s)\mid\le C_2\mid t-s\mid^v, \quad t,s\in [0,1],
	\end{eqnarray*}
	where the constant $ C_2>0 $ and $ v \in [0,1] $. In what follows, we set $ d = d' + v\ge p+1 $. 
\end{assumption}
\begin{assumption}\label{A3}
	Assume the density functions $ f_{\epsilon\mid\boldsymbol{X}(t)}(\epsilon_i,x_i(t)) $, $ i=1,2,\dots,n $, are continuous and uniformly bounded away from 0 and $ \infty $ at $ \epsilon_i=0 $. Furthermore, assume $ \mathrm{max}_{i=1,2,\dots,n} \mathrm{E}(\epsilon_i^4) < \infty $.
\end{assumption}	
\begin{assumption}\label{A4}
	Assume the smoothing parameter $ \lambda $ satisfies $ \lambda=o(n^{1/2}K^{1/2-2q}) $ with $ q\le p $.
\end{assumption}
\begin{assumption}\label{A5}
	Assume the number of knots $ K = o(n^{1/2})$ and $ K/n^{1/(2d+1)} \rightarrow \infty $ as $ n \rightarrow \infty$.
\end{assumption}

\begin{remark}\label{remark1}
	Assumptions \ref{A1} and 
	\ref{A2} 
	are quite usual in the functional setting; see e.g., 
	\citep{Cardot2005quantile,Claeskens2009linear,Yoshida2013quantile}. 
	Assumption \ref{A3} is a regular condition also used in 
	\cite{Koenker2005quantile, Cardot2005quantile} 
	and can imply the uniqueness of the conditional quantile of order $ \tau $. Assumptions \ref{A4} and 
	\ref{A5} are used to ensure the unbiasedness of  the estimator 
	\citep{liu2021functional}.
\end{remark}

To describe the asymptotic form of $ \hat{\beta}(t) $, we 
also need the following preparations. Define $ \boldsymbol{G}=\frac{1}{n}\sum_{i=1}^{n}\boldsymbol{B}_i\boldsymbol{B}^T_i $, $ \boldsymbol{G}_{\tau}=\frac{1}{n}\sum_{i=1}^{n}f_{\epsilon\mid\boldsymbol{X}(t)}(0,x_i(t))\boldsymbol{B}_i\boldsymbol{B}^T_i $ and $ \boldsymbol{H}_\tau=\boldsymbol{G}_\tau+\lambda/n\boldsymbol{D}_q $. 
Then, we have $ \Vert \boldsymbol{G} \Vert_{\infty}=O(K^{-1})$ and $ \Vert \boldsymbol{D}_q \Vert_{\infty}=O(K^{2q-1})$; see Lemma \ref{lem1} in the Appendix. Related results can also be found in 
\cite{Cardot2003linear,Claeskens2009linear,liu2021functional} 
and the references therein. Meanwhile, combining Assumptions \ref{A3} and \ref{A4}, we have $ \Vert \boldsymbol{H}^{-1}_\tau \Vert_{\infty}=O(K) $, where $\Vert \boldsymbol{A}\Vert_{\infty}=\mathrm{max}_{ij}\{\mid a_{ij}\mid\}$ for a matrix $\boldsymbol{A}=(a_{ij}) $.
Furthermore, Assumption \ref{A2} implies that there exists a spline function $ \beta_0(t)=\boldsymbol{B}^T(t)\boldsymbol{\theta}_0$, called spline approximation of $ \beta(t) $, which as $ K\rightarrow \infty $, satisfies
\begin{equation*}
	\mathop{\rm sup}\limits_{t \in [0,1]} \mid\beta(t)+b_a(t)-\boldsymbol{B}^T(t)\boldsymbol{\theta}_0\mid=o(K^{-d}),
\end{equation*}
where
\begin{eqnarray*}
	b_a(t)=-\frac{\beta^d(t)}{K^dd!}\sum_{j=0}^{K}I(t_j\le t< t_{j+1})\mathrm{Br}_d\left(\frac{t-t_j}{K^{-1}}\right)=O(K^{-d})    
\end{eqnarray*}
is the spline approximation bias with $ I(a < x < b) $ being the indicator function of an interval $ (a, b) $ and $ \mathrm{Br}_d(t) $ being the $ d$-th Bernoulli polynomial; see e.g., 
\cite{Zhou1998spline}. Thus, the penalized spline quantile estimator can be decomposed as
\begin{eqnarray*}
	\hat{\beta}(t)-\beta(t)=\hat{\beta}(t)-\beta_0(t)+\beta_0(t)-\beta(t)=\hat{\beta}(t)-\beta_0(t)+b_a(t)+o(K^{-d}).
\end{eqnarray*}

Now, we present the asymptotic distribution of $ \hat{\beta}(t)$ in the following Theorem.
\begin{theorem}\label{Th1} 
	Under the Assumptions \ref{A1}--\ref{A3}, for $ t \in [0, 1] $, as $ n \rightarrow \infty $, we have
	\begin{eqnarray}\label{2.5}
		\left\{\boldsymbol{B}(t)^T\boldsymbol{V}_0\boldsymbol{B}(t)\right\}^{-1/2}\sqrt{n/K}\left(\hat{\beta}(t)-\beta(t)-b_a(t)-b_{\lambda}(t)\right)\rightarrow N(0,1),
	\end{eqnarray}
	where the shrinkage bias is define as
	\begin{eqnarray*}
		b_{\lambda}(t)=-\frac{\lambda}{n}\boldsymbol{B}^{T}(t)\boldsymbol{H}^{-1}_\tau \boldsymbol{D}_q\boldsymbol{\theta}_0=O(\lambda K^{2q}/n), 
	\end{eqnarray*}
	and $ \boldsymbol{V}_0 $ is the asymptotic variance-covariance of $ \sqrt{n/K}(\hat{\boldsymbol{\theta}}-\boldsymbol{\theta}_0)$ and is given as 
	\begin{eqnarray*}
		\boldsymbol{V}_0=\frac{\tau(1-\tau)}{K}\boldsymbol{H}^{-1}_\tau \boldsymbol{G}\boldsymbol{H}^{-1}_\tau =O(1).
	\end{eqnarray*}	
\end{theorem}

Since Assumption \ref{A5} ensures that the order of $ K $ is $ n^v $, where $ v\ge 1/(2d+1) $, 
the spline approximation bias $ b_a(t)=O(K^{-d}) $ is negligible. In addition, from Assumption \ref{A4}, we can get $ b_{\lambda}(t)=o(\sqrt{K/n}) $. Thus the shrinkage bias is also negligible.
By the above discussion, we have the following theorem.
\begin{theorem}\label{Th2}
	Under the Assumptions \ref{A1}--\ref{A5}, for $ t \in [0, 1] $, as $ n \rightarrow \infty $,
	\begin{eqnarray*}
		\{\boldsymbol{B}(t)^T\boldsymbol{V}_0\boldsymbol{B}(t)\}^{-1/2}\sqrt{n/K}(\hat{\beta}(t)-\beta(t))\rightarrow N(0,1),
	\end{eqnarray*}
	where $ \boldsymbol{V}_0 $ is given in Theorem \ref{Th1}.   
\end{theorem}

	\section{The optimal subsampling}\label{sec.3}
\subsection{Subsampling estimator and its asymptotic distribution}\label{sec.3.1}

We first introduce a random subsampling approach, in which subsamples are taken at random with replacement based on some sampling distributions. Let $ R_i $ be the total number of times that the $ i $-th data point is selected from the full data in a subsample and $\sum_{i=1}^{n}R_i=r$, which is 
carried out by using a random subsampling method with the probabilities $ \pi_i $, $ i = 1,\dots,n $, such that $\sum_{i=1}^{n}\pi_i=1 $. Each $ R_i $ has a binomial distribution $ \rm{Bin}\mathnormal{(r,\pi_i)} $ since we use subsampling with replacement. 
Because $ \pi_i $ may depend on the full data $ \mathcal{F}_n = \{(x_i(t), y_i), i = 1,\dots,n,t\in[0,1]\}$, we need to add inverses of $ \pi_i$'s as weights to the objective function of the subsample to guarantee that the loss function is unbiased.
Thus, the subsampling estimator of the spline coefficient vector, says $\boldsymbol{\tilde{\theta}}$, is determined by minimizing 
\begin{eqnarray}\label{3.1}
	L^{\ast}(\boldsymbol{\theta};\lambda,K)=\frac{1}{r}\sum_{i=1}^{n}\frac{R_i\rho_\tau(y_i-\boldsymbol{B}^ T_i\boldsymbol{\theta})}{\pi_i}+\frac{\lambda}{2}\boldsymbol{\theta}^T\boldsymbol{D}_q\boldsymbol{\theta}.		
\end{eqnarray}

Now we investigate the asymptotic properties of $ \tilde{\beta}(t)=\boldsymbol{B}^T(t)\boldsymbol{\tilde{\theta}} $ under Assumptions \ref{A6} listed below, which 
restricts the weights in the loss function (\ref{3.1}) and hence 
can be used to protect the loss function from inflating greatly by data points with extremely small subsampling probabilities. This assumption is also required in 
\cite{Ai2021optimal} and 
\cite{liu2021functional}.
\begin{assumption}\label{A6}
	Assume that $ \mathrm{max}_{i=1,\dots,n} (n\pi_i)^{-1} = O(r^{-1})$ and $ r=o(K^2) $.
\end{assumption}

\begin{theorem}\label{Th3} 
	Under the Assumptions \ref{A1}--\ref{A6}, letting $ \eta=\mathrm{lim}_{n\rightarrow \infty} r/n $, for $ t \in [0, 1] $, as $ r,n\rightarrow \infty $, we have
	\begin{eqnarray*}
		\left\{\boldsymbol{B}(t)^T\boldsymbol{V}\boldsymbol{B}(t)\right\}^{-1/2}\sqrt{r/K}\left(\tilde{\beta}(t)-\beta(t)\right)\rightarrow N(0,1),
	\end{eqnarray*}
	in distribution, where
	\begin{eqnarray*}
		\boldsymbol{V}=\frac{\tau(1-\tau)}{K}\boldsymbol{H}^{-1}_\tau (\boldsymbol{V}_\pi+\eta \boldsymbol{G})\boldsymbol{H}^{-1}_\tau ,\quad \boldsymbol{V}_\pi=\frac{1}{n^2}\sum_{i=1}^{n}\frac{\boldsymbol{B}_i\boldsymbol{B}^T_i}{\pi_i}.
	\end{eqnarray*}	        
\end{theorem}

\subsection{Optimal subsampling probabilities}\label{sec.3.2}

To better approximate $\beta(t)$, it is important to choose the proper subsampling probabilities. It would be meaningful if the asymptotic integrated mean squared error (IMSE) of $ \tilde{\beta}(t) $ attains its minimum. By Theorem \ref{Th3} and observing that 
$ \tilde{\beta}(t) $ is asymptotically unbiased, we have the asymptotic IMSE of $ \tilde{\beta}(t) $ as follows
\begin{eqnarray}\label{3.2}
	IMSE(\tilde{\beta}(t)-\beta(t))=\frac{K}{r}\int_{0}^{1}\boldsymbol{B}^T(t)\boldsymbol{V}\boldsymbol{B}(t)\mathrm{d}t.
\end{eqnarray}
Note that, in (\ref{3.2}), $ \boldsymbol{V} $ is the asymptotic variance-covariance matrix of $ \sqrt{r/K}(\boldsymbol{\tilde{\theta}}-\boldsymbol{\theta}_0)$ and the integral inequality $ \int_{0}^{1}\boldsymbol{B}^T(t)\boldsymbol{V}\boldsymbol{B}(t)\mathrm{d}t \le \int_{0}^{1}\boldsymbol{B}^T(t)\boldsymbol{V'}\boldsymbol{B}(t)\mathrm{d}t $ holds if and only if $ \boldsymbol{V}\le \boldsymbol{V'} $ holds in the Lowner-ordering  sense. Thus, we focus on minimizing the asymptotic variance-covariance matrix $\boldsymbol{V}$ and choose the subsampling probabilities such that $\mathrm{tr}(\boldsymbol{V})$ is minimized. This is called the A-optimality criterion in optimal experimental design; see e.g., 
\cite{Atkinson2007}. Using this criterion, we are able to derive an explicit expression of optimal subsampling probabilities in the following theorem.
\begin{theorem}[A-optimality]\label{Th4} If the subsampling probabilities $\pi_i, i=1, \dots,n,$ are chosen as
	\begin{eqnarray}\label{3.3}
		\pi_i^{FAopt}=\frac{\Vert \boldsymbol{H}^{-1}_{\tau}\boldsymbol{B}_i\Vert_2}{\sum_{i=1}^{n}\Vert \boldsymbol{H}^{-1}_{\tau}\boldsymbol{B}_i\Vert_2},
	\end{eqnarray}
	then the total asymptotic MSE of $ \sqrt{r/K}(\boldsymbol{\tilde{\theta}}-\boldsymbol{\theta}_0)$, $\mathrm{tr}(\boldsymbol{V})$, attains its minimum, and so does 
	the asymptotic IMSE of $ \tilde{\beta}(t) $. 
\end{theorem}

However, $ \boldsymbol{H}_{\tau} $ in (\ref{3.3}) depends on the density functions of $\epsilon_i\ (i=1,\dots, n)$
at zero given the respective $ x_i(t) $ and hence the implementation of this subsampling method requires reasonable estimation for all the density functions $ f_{\epsilon\mid\boldsymbol{X}(t)}(0,x_i(t))$, which are often infeasible 
in practice without additional information. In addition, it also requires the chosen of smoothing parameter $ \lambda $ in $ \boldsymbol{H}_{\tau} $ and the calculation of $ \Vert \boldsymbol{H}^{-1}_{\tau}\boldsymbol{B}_i\Vert_2 $ for $ i = 1, 2, \dots, n $, which costs
$ O(n(K+p+1)^2) $. These weaknesses make this optimal subsampling method not suitable for practical use. 
While, for the independent identically distributed (i.i.d.) errors case,  the $ \boldsymbol{G}_{\tau} $ in $ \boldsymbol{H}_{\tau} $ can be simply replaced by $ f_{\epsilon\mid\boldsymbol{X}(t)}(0,x(t))\boldsymbol{G}$ since $ f_{\epsilon\mid\boldsymbol{X}(t)}(0,x_i(t))=f_{\epsilon\mid\boldsymbol{X}(t)}(0,x(t)) $ for all $ i $. 

As observed in (\ref{3.3}), only $ \boldsymbol{V}_\pi $ involves $\pi_i$ in the asymptotic variance-covariance matrix $\boldsymbol{V}=\tau(1-\tau)K^{-1}\boldsymbol{H}^{-1}_\tau (\boldsymbol{V}_\pi+\eta \boldsymbol{G}) \boldsymbol{H}^{-1}_\tau $, and  $\boldsymbol{H}^{-1}_\tau \boldsymbol{V}_\pi \boldsymbol{H}^{-1}_\tau\le \boldsymbol{H}^{-1}_\tau \boldsymbol{V}_{\pi'}\boldsymbol{H}^{-1}_\tau $ if and only if $ \boldsymbol{V}_\pi \le \boldsymbol{V}_{\pi'} $ in the Lowner-ordering. Thus, we focus on $ \boldsymbol{V}_\pi $ and choose to minimize its trace. which can be interpreted as minimizing the asymptotic MSE of $ \sqrt{r/K}\boldsymbol{H}_{\tau}(\boldsymbol{\tilde{\theta}}-\boldsymbol{\theta}_0) $ due to its asymptotic unbiasedness. This is called L-optimality criterion in optimal experimental design 
\citep{Atkinson2007}. Therefore, to circumvent density function estimation and save calculation cost, we consider the modified optimal criterion: minimizing $\mathrm{tr}(\boldsymbol{V}_\pi) $.
\begin{theorem}[L-optimality]\label{Th5} If the subsampling probabilities $\pi_i, i=1, \dots,n,$ are chosen as
	\begin{eqnarray}\label{3.4}
		\pi_i^{FLopt}=\frac{\Vert \boldsymbol{B}_i\Vert_2}{\sum_{i=1}^{n}\Vert \boldsymbol{B}_i\Vert_2},
	\end{eqnarray}
	then $\mathrm{tr}(\boldsymbol{V}_\pi)$ attains its minimum.
\end{theorem}

The functional L-optimal subsampling probabilities $\pi^{FLopt}_i $ do not depend on the densities of $\epsilon_i$ given the respective $ x_i(t) $, and thus are much easier to implement compared with the functional A-optimal subsampling probabilities $ \pi^{FAopt}_i $. In addition, $ \pi^{FLopt}_i $ requires $O(n(K+p+1))$ flops to compute, which is much cheaper than $ \pi^{FAopt}_i $ as $ K $ increases.

Furthermore, it is worth noting that the subsampling probabilities $\pi^{FLopt}_i\  (i=1,\dots, n)$ in (\ref{3.4}) do not contain responses and do not depend on the covariates directly. In fact, the structural information of the covariates is described by the expression $ \Vert \boldsymbol{B}_i\Vert_2= \Vert\int_{0}^{1}x_i(t)\boldsymbol{B}(t)\mathrm{d}t \Vert_2$, which is similar to the statistical leverage score. As a result, the subsampling probabilities 
result in the non-informative sampling. This allows us to try different models based on the subsamples. It is in contrast to the subsampling probabilities used in functional linear regression, which result in 
the informative sampling 
\citep{liu2021functional}.



\subsection{Tuning parameter selection}\label{sec.3.3}

There are four parameters in estimation of $ \beta(t) $: the number of knots $ K $, the degree $ p $ for spline functions, the smoothing parameter $ \lambda $ and the order of derivation $ q $ for the estimator. However, the number of knots $ K $ is not a crucial parameter because smoothing is controlled by the roughness penalty parameter $ \lambda $; see e.g., 
\cite{Ruppert2002knot,Cardot2003linear}. 
In addition, the degree of spline functions $ p $ and the order of derivatives $ q $ are also 
known to be less important. This is because, in practice, we usually smooth with B-splines of degree 3 and a second-order penalty. Once other parameters are fixed, a natural way to determine the parameter $ \lambda $ is to minimize a leave-one-out cross-validation criterion. We preferably employ the generalized approximate cross-validation (GACV) criterion introduced by 
\cite{Yuan2006GACV} in smoothing splines problems, which is defined by
\begin{eqnarray*}
\mathrm{GACV}(\lambda)=\frac{\sum_{i=1}^{n}\rho_\tau(y_i-\boldsymbol{B}_i^T\boldsymbol{\hat{\theta}})}{n-\mathrm{df}_{\lambda}},
\end{eqnarray*}
where $ \mathrm{df}_{\lambda} $ denotes the effective degrees of freedom of the fit. In the present paper, we implement  $ \boldsymbol{\hat{\theta}}=(\boldsymbol{B}^T\boldsymbol{WB} + \lambda \boldsymbol{D})^{-1}\boldsymbol{B}^T\boldsymbol{W}\boldsymbol{y}$ and $ \mathrm{df}_{\lambda}=\mathrm{tr}\left(\boldsymbol{B}(\boldsymbol{B}^T\boldsymbol{WB} + \lambda \boldsymbol{D})^{-1}\boldsymbol{B}^T\boldsymbol{W}\right) $ in the penalized iteratively reweighted least squares (PIRLS) method which is useful to solve the functional quantile regression problem; see e.g., 
\cite{Cardot2005quantile,reiss2012}. 
In the above expressions, $ \boldsymbol{W} $ is a diagonal matrix whose diagonal elements are weights,
\begin{eqnarray*}
w_i^{(k)}=\frac{\tau-I[y_i-\boldsymbol{B}_i^T\boldsymbol{\hat{\theta}}^{(k)}]}{2[y_i-\boldsymbol{B}_i^T\boldsymbol{\hat{\theta}}^{(k)}]},\quad i=1,2,\dots,n,
\end{eqnarray*}
which are iterated until convergence; see Appendix A of 
\cite{reiss2012}. However, using full data to select the optimal $ \lambda $ is computationally expensive, so we 
select the smoothing parameter $ \lambda $ by GACV under the optimal subsample data.

\section{Numerical Experiments}
\label{sec.4}

In this section, we aim to study the finite sample performance of the proposed methods 
by using synthetic and real data. 
\subsection{Simulation}\label{sec.4.1}

We generated the functional covariates in a similar way to that adopted in 
\cite{liu2021functional}. More specifically, the functional covariates were identically and independently generated as: 
\begin{eqnarray*}
	x_i(t)=\sum a_{ij}\boldsymbol{B}_j(t), \quad i=1,2,\dots,n,
\end{eqnarray*}
where $ \boldsymbol{B}_j(t) $ are cubic B-spline basis functions that are sampled at 100 equally spaced points between 0 and 1. We consider the following three different distributions for the basis coefficient $ \boldsymbol{A}=(a_{ij}) $:
\begin{enumerate}[\hspace{2em}(1)]
	\item  \textbf{mvNormal}. Multivariate normal distribution $ N(\boldsymbol{0}, \boldsymbol{\Sigma}) $, where $ \boldsymbol{\Sigma}_{ij}=0.5^{\mid i-j\mid} $;
	\item \textbf{mvT3}. Multivariate $ t $ distribution with 3 degree of freedom, $t_3(\boldsymbol{0},\boldsymbol{\Sigma})$;
	\item  \textbf{mvT2}. Multivariate $ t $ distribution with 2 degree of freedom, $t_2(\boldsymbol{0},\boldsymbol{\Sigma})$.
\end{enumerate}
The responses are generated as following:
\begin{eqnarray*}
	y_i=\int ^1_0 x_i(t)\beta(t)\mathrm{d}t+\epsilon_i,\quad i=1,2,\dots,n,
\end{eqnarray*}
where the slope function $ \beta(t)=2t^2+0.25t+1 $ and the random errors, $ \epsilon_i $'s, are generated in three cases:
\begin{enumerate}[\hspace{2em}(1)]
	\item  \textbf{Normal}. The standard normal distribution;
	\item  \textbf{T1}. $ t_1 $ distribution ;
	\item  \textbf{Hetero}. The standard normal distribution times $\int ^1_0 \mid x_i(t)(t+1)\mid \mathrm{d}t $.
\end{enumerate}
The first two designs consider symmetric i.i.d. random errors while the last one considers conditional heteroscedastic errors. 

We first take $ n=10^5 $ for training, $ m = 1000 $ for testing and $ \tau=0.5,0.75 $ to investigate the influence of different quantile level on performance of the proposed subsampling methods. From Assumption \ref{A5}, we let the number of knots $ K= \lceil n^{1/4}\rceil$. We shall compare the functional A-optimal subsampling (FAopt) and L-optimal subsampling (FLopt) methods with the uniform subsampling  (Unif) method. For fair comparison, we use the same basis functions and the same smoothing parameter in the three methods with the same full data. For each $ \tau $, we will compute the root integrated mean squared error (IMSE) from 1000 repetitions:
\begin{eqnarray*}
	\rm{IMSE}=\frac{1}{1000}\sum_{k=1}^{1000}\sqrt{\int ^1_0 \left\{\tilde{\beta}^{(k)}(t)-\beta(t)\right\}^2\mathrm{d}t},
\end{eqnarray*}
where $ \tilde{\beta}^{(k)}(t) $ is the estimator from the $ k $-th run. All the experiments are implemented in R programming language on a PC with an Intel I5 processor and 16GB memory. 	

\begin{figure}[t]
	\centering
		\includegraphics[width=0.95\textwidth]{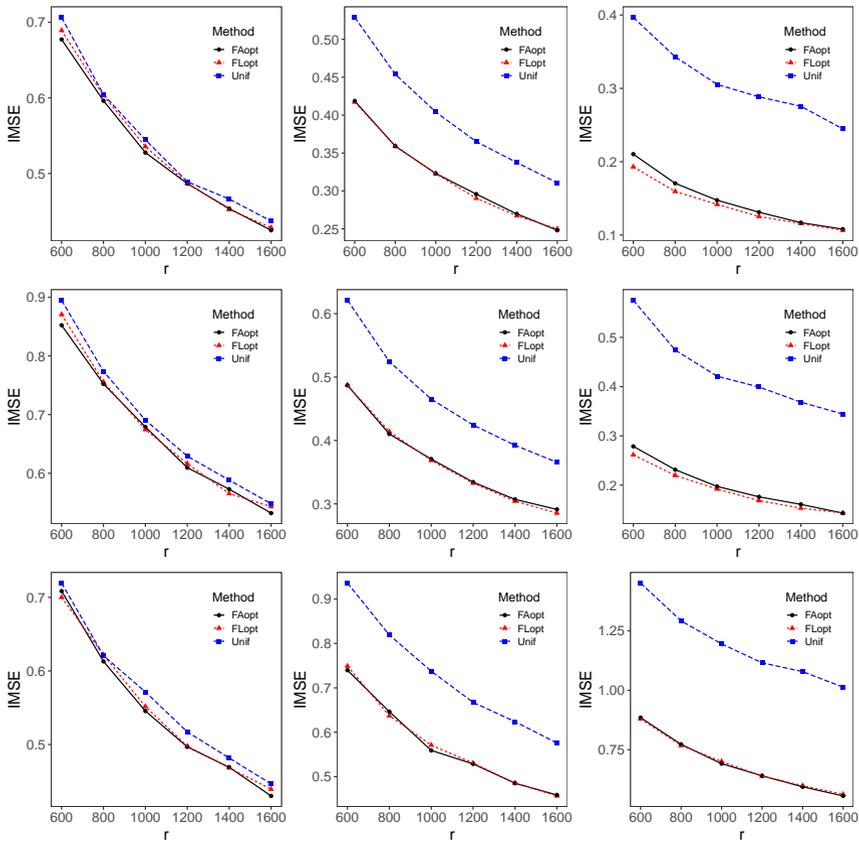}	
	\caption{IMSE for different subsampling size $ r $ with different distributions when $ \tau=0.5 $ and $ n=10^5 $.}\label{fig1}
\end{figure}
\begin{figure}[h]
	\centering
		\includegraphics[width=0.95\textwidth]{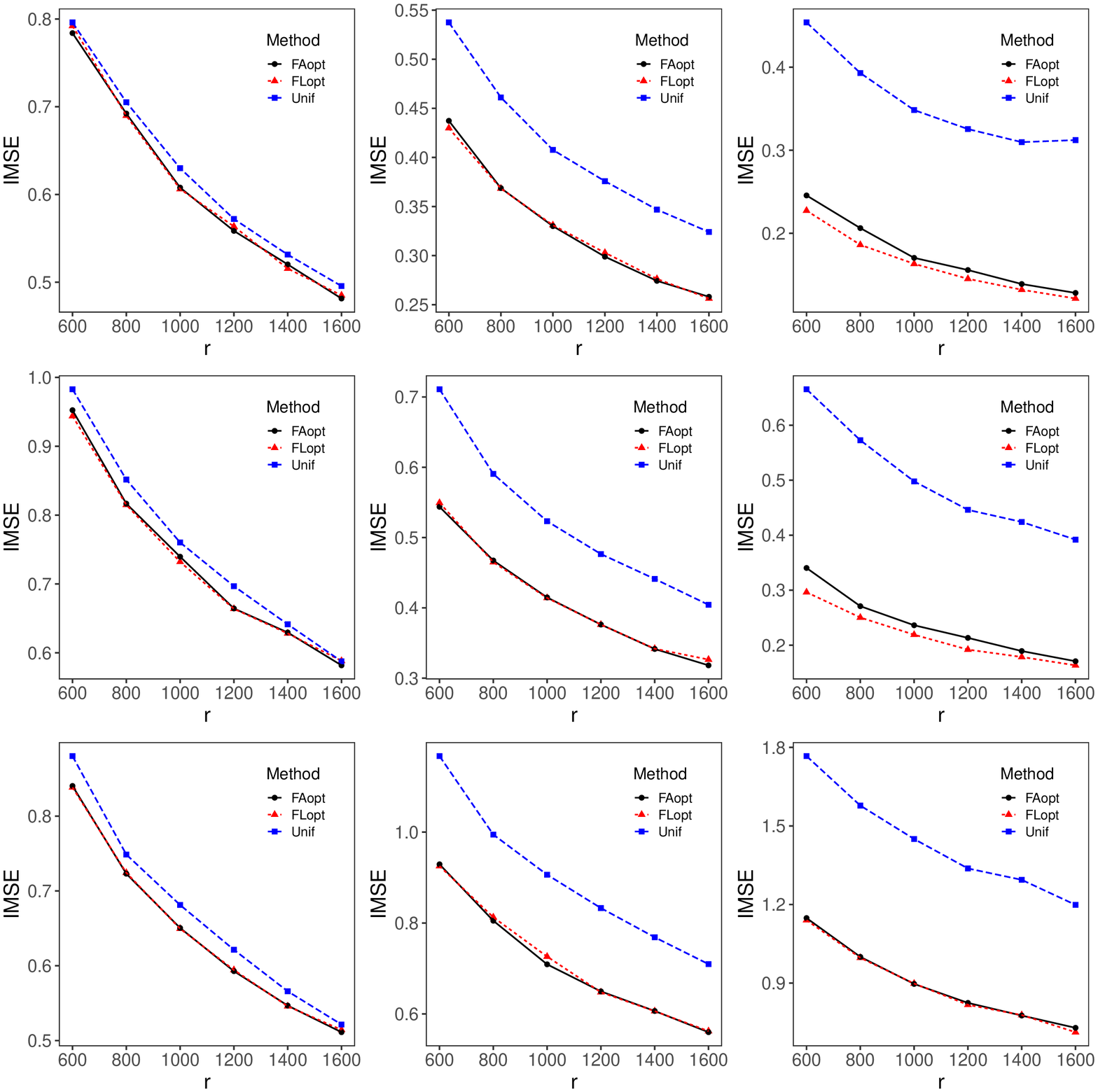}	
		\caption{IMSE for different subsampling size $ r $ with different distributions when $ \tau=0.75 $ and $ n=10^5 $}\label{fig2}
	\end{figure}

Figures \ref{fig1} and  \ref{fig2} display the simulation results corresponding to various subsampling sizes of 600, 800, 1000, 1200, 1400 and 1600 under different quantile level\footnote {In Figures \ref{fig1}, \ref{fig2}, and \ref{fig3}, the three columns correspond to the three distributions of the basis coefficients (mvNormal, mvT3, mvT2), respectively, and the three rows correspond to the three distributions of random errors (Normal, T1, Hetero), respectively. 
}. 
It is clear to see that the FAopt and FLopt subsampling methods always have smaller IMSEs than the Unif subsampling method for all cases, which is in agreement with the theoretical results that they aim to minimize the asymptotic IMSEs of the subsampling estimator. Moreover, the advantages of the FAopt and FLopt subsampling methods become more significant as the tail of the basis coefficient distribution becomes heavier. Besides, we also see that the FAopt and FLopt methods tend to perform similarly, even though the fact that the FLopt method does not theoretically minimize the MSE of the subsample spline coefficient $\boldsymbol{\tilde{\theta}} $. 

To further assess the relative performance of the proposed methods in comparison with the full data estimator, the prediction efficiency (PE) is adopted on the test data of simulation, which is defined as follows:
\begin{eqnarray*}
	PE=\frac{\sum_{i}\left[\int ^1_0 x_i(t)\beta(t)\mathrm{d}t-\int ^1_0 x_i(t)\tilde{\beta}(t)\mathrm{d}t\right]^2}{\sum_{i} \left[\int ^1_0 x_i(t)\beta(t)\mathrm{d}t-\int ^1_0 x_i(t)\hat{\beta}(t)\mathrm{d}t\right]^2},\quad i\in{\rm testset}.
\end{eqnarray*}
We plot the logarithm of prediction efficiency for the FAopt, FLopt and Unif methods when $ \tau=0.75 $ in Figure \ref{fig3}, from which we can see that the FAopt and FLopt methods significantly outperform the Unif method, and the FLopt method has comparable or slightly smaller prediction efficiency than the FAopt method. Results for the case $ \tau=0.5 $ are similar and thus are omitted.    

\begin{figure}[t]
	\centering
		\includegraphics[width=0.95\textwidth]{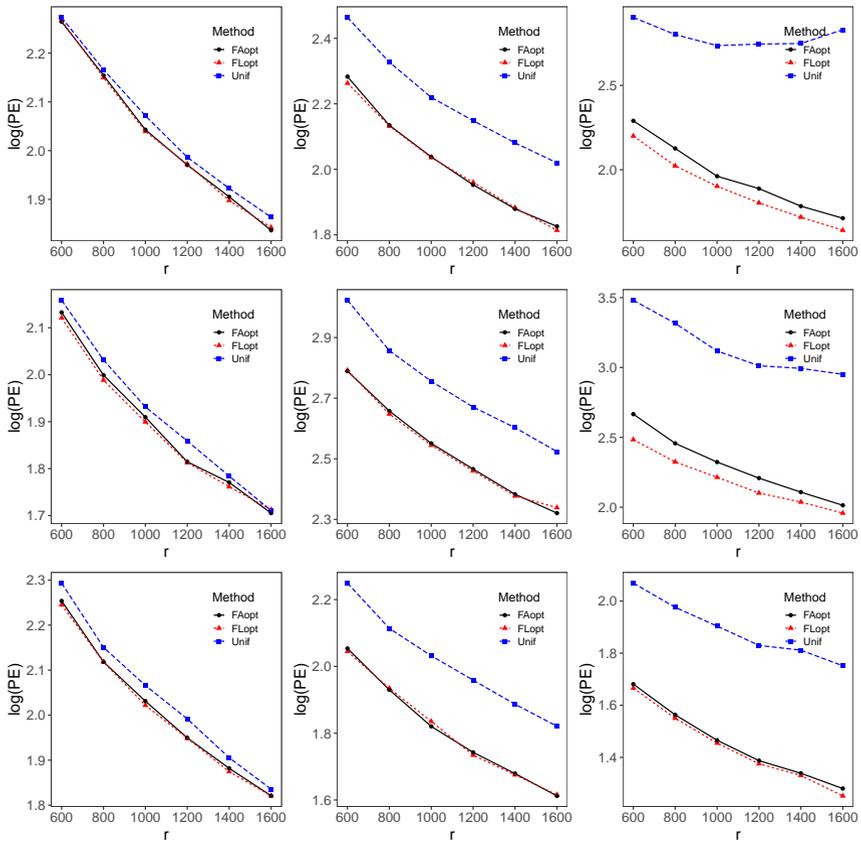}
		\caption{Log prediction efficiency for different subsampling size $ r $ with different distributions when $ \tau=0.75 $ and $ n=10^5 $ for 1000 repetitions.}\label{fig3}
	\end{figure}
To evaluate the computational efficiency of the subsampling methods, we record the computing time of the three subsampling methods. 
We use the function \textbf{Sys.time()} 
to count start and end times of the corresponding code only for the estimated part of $ \boldsymbol{\tilde{\theta}} $. Since all the cases have similar performance, we only show the results of mvNormal - Normal datasets here. The results on different $ r $ for the FAopt, FLopt and Unif subsampling methods with $ \tau=0.75 $ and $ n=10^5$ are given in Table \ref{tab1}. It is not surprising to find that the Unif method takes the least time because it does not need to calculate the additional optimal subsampling probabilities. As we expected, the FLopt method is faster than the FAopt method, which agrees with the theoretical analysis. The computing time for using full data is also given in the last row of table \ref{tab1}, which is the longest one and confirms that our proposed methods can reduce the computational burden.

\begin{table}
	\begin{center}
		\begin{minipage}{\textwidth}
			\caption{CPU seconds for different subsampling size $ r $ with $ \tau= 0.75 $ and $ n=10^5 $ for 1000 repetitions.}\label{tab1}%
			\begin{tabular*}{\textwidth}{@{\extracolsep{\fill}}ccccccc@{\extracolsep{\fill}}}
				\toprule
				\multirow{2}{*}{Method} & \multicolumn{6}{c}{$ r $}\\
				\cline{2-7}
				&600 & 800 & 1000 & 1200 & 1400 & 1600\\		
				\midrule
				FLopt &0.155 & 0.166 & 0.180 & 0.201 & 0.215 & 0.227\\
				FAopt &0.472 & 0.462 & 0.469 & 0.496 & 0.515 & 0.533\\
				Unif &0.115 & 0.133 & 0.142 & 0.161 & 0.178 & 0.193	\\
				\multicolumn{7}{l}{Full data CPU seconds: 4.086}\\
				\botrule
			\end{tabular*}
		\end{minipage}
	\end{center}
\end{table}

To further demonstrate the performance of our proposed methods in large datasets, we set the full data size to $ n=10^4 $, $ 10^5$, $ 10^6 $ and $ 5\times 10^6 $, respectively. In addition, 
we let $r=1000$, $ \lambda = 0.001 $ and enlarge the number of knots for spline function to $ K=50 $. Table \ref{tab2} presents the CPU seconds for repeating different subsampling methods for 500 times. The results indicate that our proposed methods can improve the computational efficiency compared with the full data, and their advantage is more significant as the full data size increases. 
For our two methods, we recommend the FLopt method for practical use.
\begin{table}
	\begin{center}
		\begin{minipage}{\textwidth}
			\caption{CPU seconds for different full data size $ n $ with $ r=1000$ when $ \tau= 0.75 $, $ K=50 $ and $\lambda=0.001 $ for 500 repetitions.}\label{tab2}%
			\begin{tabular*}{\textwidth}{@{\extracolsep{\fill}}lcccc@{\extracolsep{\fill}}}
				\toprule
				\multirow{2}{*}{Method} & \multicolumn{4}{c}{$ n $}\\
				\cline{2-5}
				&$ 10^4 $ & $ 10^5 $ & $ 10^6 $ & $ 5\times10^6 $\\		
				\midrule 			
				FLopt &0.307 & 0.431 & 0.656 & 2.666\\
				FAopt &0.355 & 0.790 & 5.415 & 33.625\\
				Unif &0.304 & 0.378 & 0.383 & 0.575 \\			
				Full &2.47 & 24.543 & 238.940 & 1668.454\\
				\botrule
			\end{tabular*}
		\end{minipage}
	\end{center}
\end{table}

\subsection{Beijing multi-site air-quality data}\label{sec.4.2}
Carbon monoxide (CO) is formed by incomplete combustion of fossil fuels and is ubiquitous in ambient air. The adverse health effects of very high CO concentrations, such as CO poisoning and cardiovascular deaths, are well documented; see e.g., \cite{Liu2018CO, Kinoshita2020CO, Chen2021CO}. 
Thus, air quality prediction is vital to management of human health, especially the respiratory system. There has been extensive research on prediction CO concentrations, see e.g., 
\cite{Moazami2016CO,Shams2020CO}. 

\begin{figure}[t]
	\centering
		\includegraphics[scale=0.6]{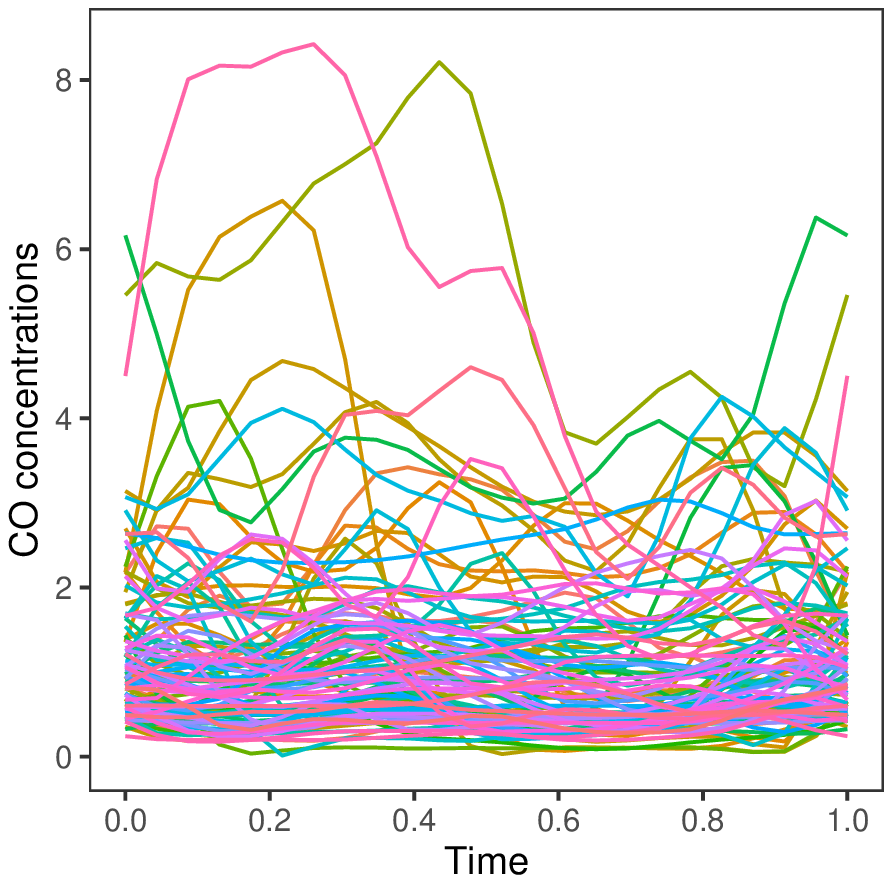}
	\quad
	\includegraphics[scale=0.6]{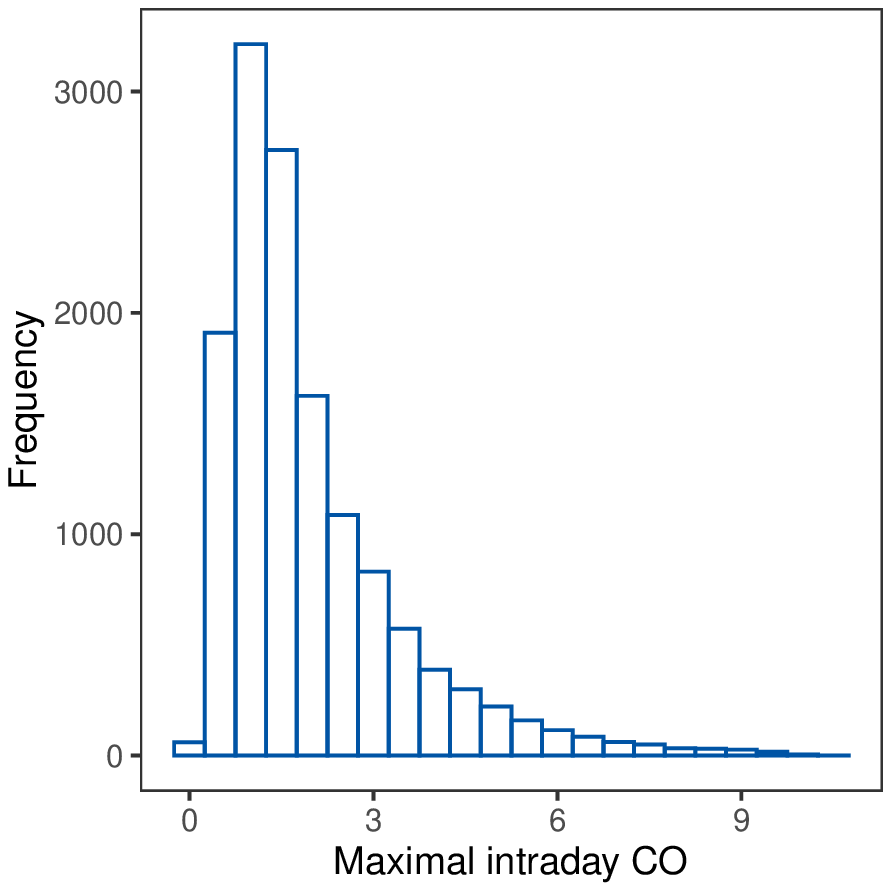}
	\caption{Left subfigure: A random subset of 100 curves of 24-hourly CO concentrations. Right subfigure: Histogram of the maximal values of intraday CO concentrations.}
	\label{fig4}
\end{figure}

Now we analyze a dataset availabled from \url{https://archive-beta.ics.uci.edu/ml/datasets/beijing+multi+site+air+quality+data}. This data set consists of hourly air pollutants data from 12 nationally controlled air-quality monitoring sites in Beijing from March 1, 2013 to February 28, 2017. Our primary interest here is to predict the maximum CO concentrations ($mg/m^3 $) using the CO trajectory (24 hour) of the last day. After removing 4001 days' records with missing values, we have a dataset of 13531 days' complete records. It is randomly partitioned into a training set of $ n=10824 $ observations and $ m=2707 $ for testing. The raw observations are first transformed into functional data using 15 Fourier basis functions. This transformation can be implemented with the \textbf{Data2fd} function in the \textbf{fda} package, suggested in 
\cite{Sang2020quantile}. A random subset of 100 curves of 24-hourly CO concentrations is presented in the left panel of Figure \ref{fig4}, where the time scale has been transformed to $ [0,1] $. The right panel of Figure \ref{fig4} further supports the fact that the covariates are heavy-tailed. It depicts the histogram of the maximal values of intraday CO concentrations.

\begin{figure}[t]
	\centering
		\includegraphics[scale=0.6]{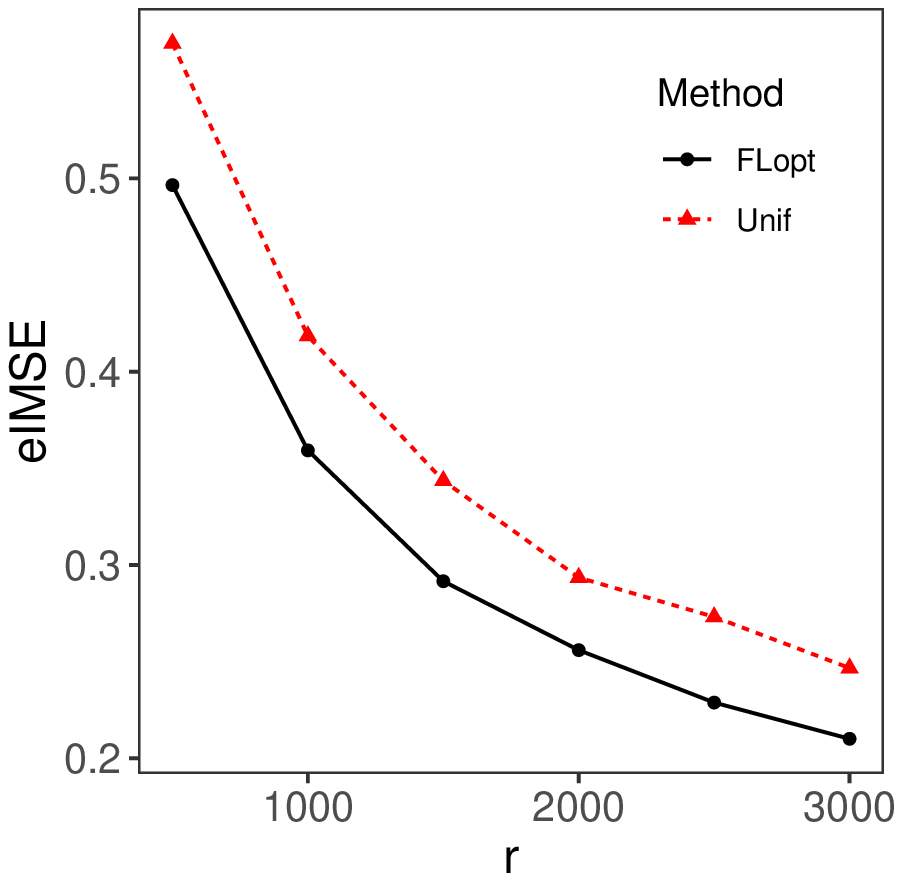}
		\quad
		\includegraphics[scale=0.6]{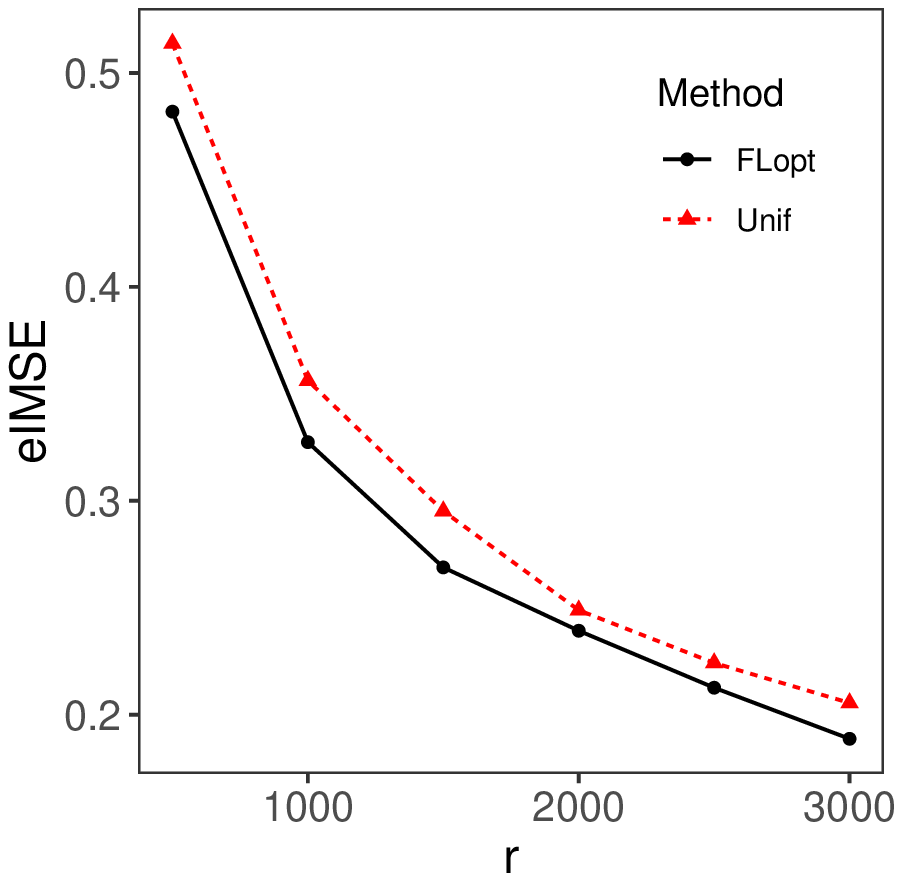}
		\caption{eIMSE for different subsampling size $ r $ when $ \tau=0.5$ (left) and 0.75 (right).}
		\label{fig5}
	\end{figure}
Since the true value $ \beta(t) $ is unknown for real dataset, we use full data estimator instead. 
We calculate the empirical IMSE using $\mathrm{eIMSE}=\frac{1}{1000}\sum_{k=1}^{1000}\sqrt{\int ^1_0 \left\{\tilde{\beta}^{(k)}(t)-\hat{\beta}(t)\right\}^2\mathrm{d}t}$, and compare the FLopt method with the Unif method. Figure \ref{fig5} shows the eIMSE of subsampling estimator for different subsampling size $ r= 500,1000,1500,2000,2500,3000$ when $\tau= 0.5$ and 0.75. We can find that the FLopt method always has smaller eIMSE than the Unif method. All eIMSEs decrease as the subsampling size $ r $ gets large, showing the estimation consistency of the subsampling method.

We further compare these two methods in terms of prediction accuracy. The relative efficiency (RE) is 
defined as follows:
\begin{eqnarray}\label{4.3}
	RE=\frac{\sum_{i}\left[\int ^1_0x_i(t)\tilde{\beta}(t)\mathrm{d}t-\int^1_0 x_i(t)\hat{\beta}(t)\mathrm{d}t\right]^2}{\sum_{i}\left[\int ^1_0x_i(t)\hat{\beta}(t)\mathrm{d}t\right]^2},\quad i\in {\rm test set}.
\end{eqnarray}
Figure \ref{fig6} displays the relative efficiency based on the subsampling method with $ \tau=0.5 $ and 0.75. In general, the relative efficiency of the subsampling estimator gradually decrease as the $ r $ increases, and the FLopt method is better than the Unif method. So it yields a better approximation to the results based on full data.
\begin{figure}[t]
	\centering
		\includegraphics[scale=0.6]{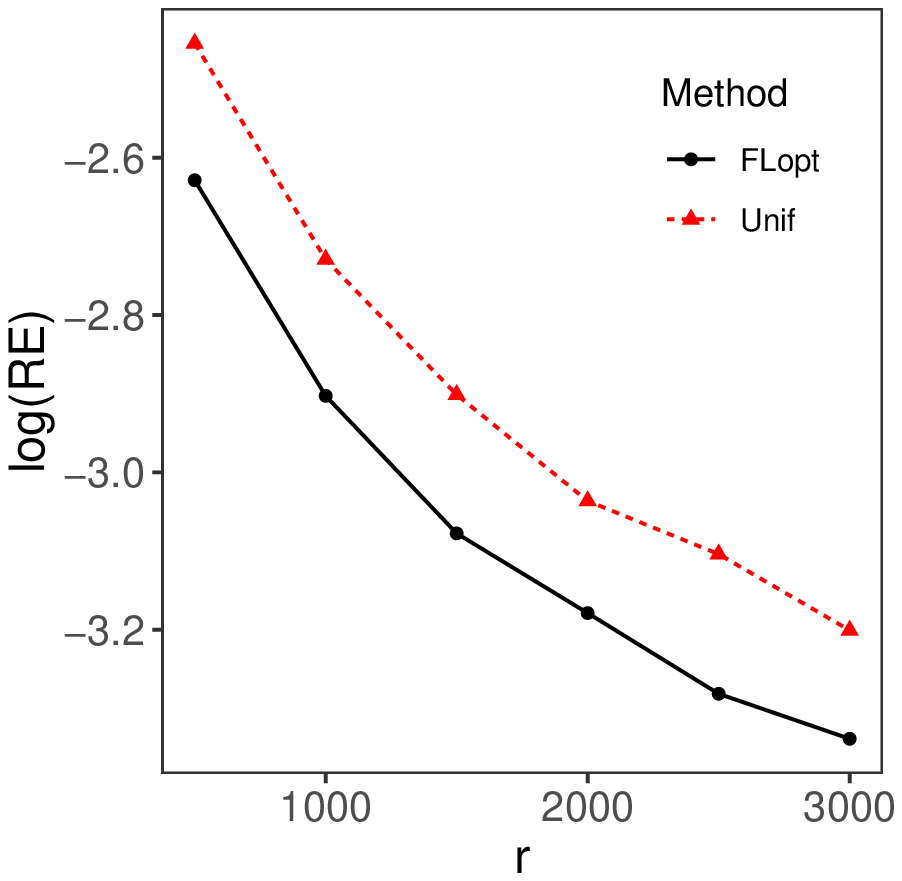}
		\quad
		\includegraphics[scale=0.6]{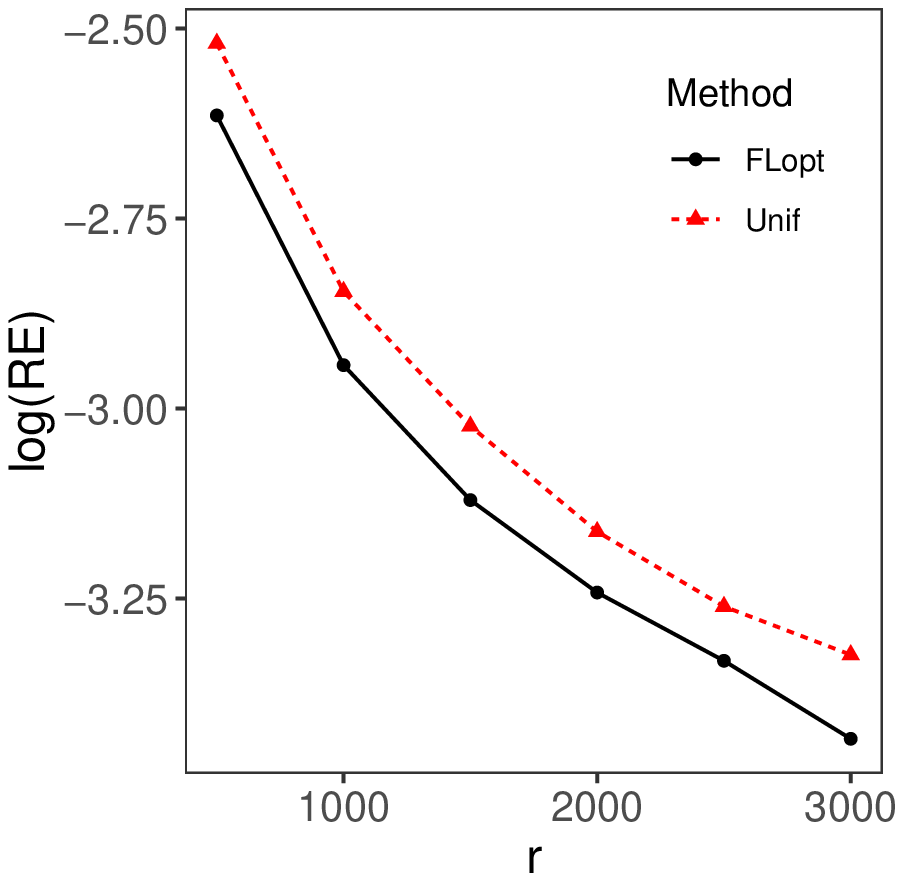}
		\caption{Log relative efficiency for different subsampling size $ r $ when $ \tau=0.5 $ (left) and 0.75 (right) for 1000 repetitions.}
		\label{fig6}
	\end{figure}

\section{Conclusions}
\label{sec.5}
Existing optimal subsampling methods mainly focus on statistic model with scalar variables or functional linear regression. 
In this paper, we develop the optimal subsampling 
for quantile regression model when the covariates are functions. Not only is asymptotic normality estimated, but also the optimal and feasible optimal subsampling probabilites are derived according to the functional A- and L-optimality criterions, respectively. The latter results in the non-informative subsampling, which is more flexible and feasible to apply to other models compared with information sampling. 
Our numerical experiments show that, the FAopt and FLopt methods outperform the Unif subsampling method and are computationally feasible for massive data, and they yield good approximations to the results based on full data.

In this paper, we only consider the subsampling for the scalar-on-function quantile regression at the single quantile level. As done in \cite{Shao2021quantile,Yuan2022quantile}, it is interesting to investigate multiple quantile level. 
Observing that our FLopt sampling probabilities are irrelevant to quantile level, this problem should be doable. In fact,  
a more interesting problem worth further investigations is how to apply optimal subsampling methods to the quantile regression process. In addition, other functional regression models are worth exploring, such as function-on-function regression and function-on-scalar regression.

\backmatter
\bmhead{Acknowledgments}
This work was supported by the National Natural Science Foundation of China (No. 11671060) and the Natural Science Foundation Project of CQ CSTC (No. cstc2019jcyj-msxmX0267).

\section*{Declarations}
The authors declare that they have no conflict of interest.

\begin{appendices}
\section{Proofs for theoretical results 
} \label{A.1}
To prove our theorems, we begin with the following several lemmas. Note that the subsampling model involves two kinds of random errors: sampling error and model error, so we need to consider these two types of randomness in the calculation		
\begin{lemma}\label{lem1} 
	Under Assumptions \ref{A1} and \ref{A5}, for any vector $ \boldsymbol{\mu} \in\mathbb{R}^{K+p+1} $, there are some positive constants $ C_3$, $C_4 $, $ C_5 $ and $ C_6 $ such that 
	\begin{eqnarray*}
		C_3K^{-1}\le\sigma_{min}(\boldsymbol{G})\le \sigma_{max}(\boldsymbol{G})\le C_4K^{-1},\\
		C_5K^{2q-1}\Vert\boldsymbol{\mu}\Vert^2_2\le\boldsymbol{\mu}^T\boldsymbol{D}_q\boldsymbol{\mu}\le C_6K^{2q-1}\Vert \boldsymbol{\mu}\Vert^2_2,
	\end{eqnarray*}
	where $ \sigma_{min}(\cdot) $ and $\sigma_{max}(\cdot) $ denote the smallest and largest eigenvalues of a matrix, respectively. In addition, we have $\Vert \boldsymbol{G}\Vert_{\infty}=O(K^{-1})$ and $\Vert \boldsymbol{D}_q\Vert_{\infty}=O(K^{2q-1})$.
\end{lemma}
\begin{proof}
 These results can be derived directly from Lemma S2 and S3 in 
\cite{liu2021functional}.
\end{proof}		
\begin{lemma}\label{lem2} 
	Under Assumptions \ref{A1}, and \ref{A3}--\ref{A5}, there are two positive constants $ C_7 $ and $ C_8 $ such that 
	\begin{eqnarray*}
		C_7K^{-1}\le\sigma_{min}(\boldsymbol{H}_{\tau})\le \sigma_{max}(\boldsymbol{H}_{\tau})\le C_8K^{-1},
	\end{eqnarray*}
	and $\Vert \boldsymbol{H}_{\tau}\Vert_{\infty}=O(K^{-1})$.
\end{lemma}		
\begin{proof}
From Assumption \ref{A3}, we have that there are two positive constants $ c_{\epsilon} $ and $ C_{\epsilon}  $ such that $ c_{\epsilon}\le f_{\epsilon\mid\boldsymbol{X}(t)}(u,x(t))\le C_{\epsilon} $. On the other hand, by Lemma \ref{lem1}, we have $\Vert\boldsymbol{G}_{\tau}\Vert_{\infty}=O(K^{-1})$. 
Thus, the lemma can be directly proved by combining Lemma \ref{lem1} with Assumption \ref{A4}.
\end{proof}	
\begin{lemma}\label{lem3} 
	Let $ \psi_\tau(u)=\tau-I(u<0) $ and $ u_i=y_i-\boldsymbol{B}^T_i\boldsymbol{\theta}_0 $. Under the same assumptions as Theorem \ref{Th3}, for any non-zero $ \boldsymbol{\delta} \in \mathbb{R}^{K+p+1}$, we have
	\begin{equation}\label{a1}
		-\sqrt{\frac{K}{r}}\sum_{i=1}^{n}\frac{R_i}{n\pi_i}\boldsymbol{B}^T_i\boldsymbol{\delta}\psi_\tau(u_i)=-\sqrt{K}\boldsymbol{W}^T\boldsymbol{\delta}+o_P(1),
	\end{equation}
	where $ \left\{\tau(1-\tau)(\boldsymbol{V}_{\pi}+\eta \boldsymbol{G})\right\}^{-1/2}\boldsymbol{W}\rightarrow N(0,1) $ in distribution.
\end{lemma}

\begin{proof}	
	Set 
\begin{eqnarray*}
	U_r=-\sqrt{\frac{K}{r}}\sum_{i=1}^{n}\frac{R_i}{n\pi_i}\boldsymbol{B}^T_i\boldsymbol{\delta}\psi_\tau(u_i).
\end{eqnarray*}
To prove the asymptotic normality of $ U_r $, it suffices to verify that $ U_r $ satisfies the Lindeberg-Feller conditions. Firstly, the conditional expectation and conditional variance are given by 
\begin{eqnarray*}
	\mathrm{E}\left\{U_r\mid\mathcal{F}_n\right\}&=&-\sqrt{\frac{K}{r}}\sum_{i=1}^{n}\mathrm{E}\left\{\frac{R_i}{n\pi_i}\boldsymbol{B}^T_i\boldsymbol{\delta}\psi_\tau(u_i)\mid\mathcal{F}_n\right\} \nonumber    \\
	&=&-\frac{\sqrt{rK}}{n}\sum_{i=1}^{n}\boldsymbol{B}^T_i\boldsymbol{\delta}\psi_\tau(u_i),\\
	\mathrm{Var}\left\{U_r\mid\mathcal{F}_n\right\}&=&\frac{K}{r}\sum_{i=1}^{n}\mathrm{Var}\left\{\frac{R_i}{n\pi_i}\boldsymbol{B}^T_i\boldsymbol{\delta}\psi_\tau(u_i)\mid\mathcal{F}_n\right\} \nonumber \\ &=&\frac{K}{n^2}\sum_{i=1}^{n}\frac{\pi_i(1-\pi_i)}{\pi^2_i}(\boldsymbol{B}^T_i\boldsymbol{\delta})^2\psi^2_\tau(u_i).
\end{eqnarray*}

From the fact that $ \mathrm{P}(y_i<\int ^1_0 x_i(t)\beta(t)\mathrm{d}t\mid x_i(t))=\tau $, we have
\begin{eqnarray*}
	\mathrm{E}\left\{\psi_\tau(u_i)\mid x_i(t)\right\}
	&=&\tau-\mathrm{E}\left\{I(u_i<0)\mid x_i(t)\right\}\nonumber    \\
	&=&\tau-\mathrm{P}(y_i<\boldsymbol{B}^T_i\theta_0\mid x_i(t))\nonumber    \\
	&=& \tau-\mathrm{P}\left(y_i<\int ^1_0 x_i(t)(\beta(t)+b_a(t)(1+o_P(1)))\mathrm{d}t\mid x_i(t)\right) \nonumber    \\
	&=& - b_i f_{\epsilon\mid\boldsymbol{X}(t)}(0,x_i(t))(1+o_P(1)) \nonumber    \\
	&=& o_P(1),
\end{eqnarray*}
where $ b_i=\int_{0}^{1}x_i(t)b_a(t)\mathrm{d}t $, and the third equality is from the definition of $ \boldsymbol{\theta}_0 $ and the fourth equality is obtained by the Taylor expansion of the cumulative distribution function of the error $ \epsilon_i $ at point $\epsilon_i=0 $. As a result, the unconditional expectation of $ U_r $ can be calculated as
\begin{eqnarray}\label{a2}
	\mathrm{E}\left[U_r\right]&=&-\frac{\sqrt{rK}}{n}\mathrm{E}\left\{\sum_{i=1}^{n}\boldsymbol{B}^T_i\boldsymbol{\delta}\psi_\tau(u_i)\mid x_i(t)\right\}\nonumber    \\
	&=& \frac{\sqrt{rK}}{n}\sum_{i=1}^{n}\boldsymbol{B}^T_i\boldsymbol{\delta} b_if_{\epsilon\mid\boldsymbol{X}(t)}(0,x_i(t))(1+o_P(1)) \nonumber    \\
	&=& O(\sqrt{rK}K^{-(d+1)}).
\end{eqnarray}
More specifically, since $ x_i(t) $ are square integrable functions, by the Cauchy-Schwarz inequality in integral form, there exist constant $ c $ such that
\begin{eqnarray*} 
	\boldsymbol{B}^2_i&=&\left(\int_0^1x_i(t)\boldsymbol{B}(t)\mathrm{d}t\right)^2\\
	&\le&\int_0^1x^2_i(t)\mathrm{d}t\cdot \int_0^1\boldsymbol{B}^2(t)\mathrm{d}t\le c\int_0^1\boldsymbol{B}^2(t)\mathrm{d}t.
\end{eqnarray*}
Similarly, we have
\begin{eqnarray*}
	b_i^2&\le&c\int_0^1b_a^2(t)\mathrm{d}t. 
\end{eqnarray*}
Thus, by the property of B-spline function, $ \int_0^1\boldsymbol{B}(t)\mathrm{d}t = O(K^{-1}) $, and $ b_a(t)=O(K^{-d})$, we can find that $ \Vert\boldsymbol{B}_i \Vert_{\infty}= O(K^{-1})$ and $ b_i=O(K^{-d})$ are satisfied. Putting them together, we obtain 
(\ref{a2}). 

On the other hand, according to total expectation formula, the unconditional variance is given by
\begin{eqnarray}\label{a3}
	\mathrm{Var}\left[U_r\right]&=&\mathrm{Var}\left\{-\sqrt{\frac{K}{r}}\sum_{i=1}^{N}\frac{R_i}{n\pi_i}\boldsymbol{B}^T_i\boldsymbol{\delta}\psi_\tau(u_i)\right\}\nonumber    \\
	&=& \mathrm{E}\left\{\mathrm{Var}\left\{-\sqrt{\frac{K}{r}}\sum_{i=1}^{n}\frac{R_i}{n\pi_i}\boldsymbol{B}^T_i\boldsymbol{\delta}\psi_\tau(u_i)\mid\mathcal{F}_n\right\}\right\} \nonumber    \\
	&\;&+\mathrm{ Var}\left\{\mathrm{E}\left\{-\sqrt{\frac{K}{r}}\sum_{i=1}^{n}\frac{R_i}{n\pi_i}\boldsymbol{B}^T_i\boldsymbol{\delta}\psi_\tau(u_i)\mid\mathcal{F}_n\right\}\right\}.
\end{eqnarray}
We first deal with the first term in (\ref{a3}) as follows
\begin{align}\label{a4}
	&\mathrm{E}\left\{\mathrm{Var}\left\{-\sqrt{\frac{K}{r}}\sum_{i=1}^{n}\frac{R_i}{n\pi_i}\boldsymbol{B}^T_i\boldsymbol{\delta}\psi_\tau(u_i)\mid\mathcal{F}_n\right\}\right\} \nonumber    \\
	=& \frac{K}{n^2}\mathrm{E}\left\{\sum_{i=1}^{n}\frac{\pi_i(1-\pi_i)}{\pi^2_i}(\boldsymbol{B}^T_i\boldsymbol{\delta})^2\psi^2_\tau(u_i)\mid x_i(t)\right\}\nonumber    \\
	=&K\tau(1-\tau)\boldsymbol{\delta}^T\left\{\sum_{i=1}^{n}\frac{\boldsymbol{B}_i\boldsymbol{B}^T_i}{n^2\pi_i}-\sum_{i=1}^{n}\frac{\boldsymbol{B}_i\boldsymbol{B}^T_i}{n^2}\right\}\boldsymbol{\delta}(1+o_P(1)).
\end{align}
Similarly, the second term in (\ref{a3}) equals
\begin{align}\label{a5}
	&\mathrm{Var}\left\{\mathrm{E}\left\{-\sqrt{\frac{K}{r}}\sum_{i=1}^{n}\frac{R_i}{n\pi_i}\boldsymbol{B}^T_i\boldsymbol{\delta}\psi_\tau(u_i)\mid\mathcal{F}_n\right\}\right\}\nonumber\\
	=& \frac{rK}{n^2}\mathrm{Var}\left\{\sum_{i=1}^{n}\boldsymbol{B}^T_i\boldsymbol{\delta}\psi_\tau(u_i)\mid x_i(t)\right\}\nonumber   \\
	=&rK\tau(1-\tau)\boldsymbol{\delta}^T\left(\sum_{i=1}^{n}\frac{\boldsymbol{B}_i\boldsymbol{B}^T_i}{n^2}\right)\boldsymbol{\delta}(1+o_P(1)).
\end{align}
Thus, substituting  (\ref{a4}) and (\ref{a5}) into (\ref{a3}), we have
\begin{eqnarray}\label{a6}
	\mathrm{Var}\left[U_r\right]&=& K\tau(1-\tau)\boldsymbol{\delta}^T\left\{\sum_{i=1}^{n}\frac{\boldsymbol{B}_i\boldsymbol{B}^T_i}{n^2\pi_i}+\frac{r-1}{n}\sum_{i=1}^{n}\frac{\boldsymbol{B}_i\boldsymbol{B}^T_i}{n}\right\}\boldsymbol{\delta}(1+o_P(1))\nonumber    \\
	&=&K\tau(1-\tau)\boldsymbol{\delta}^T\left(\boldsymbol{V}_{\pi}+\eta \boldsymbol{G}\right)\boldsymbol{\delta}(1+o_P(1)).
\end{eqnarray}

 Denote $\xi_i = -\sqrt{\frac{K}{r}}\frac{R_i}{n\pi_i}\boldsymbol{B}^T_i\boldsymbol{\delta}\psi_\tau(u_i)$. We now check the Lindeberg-Feller conditions. For every $ \epsilon>0 $,
\begin{align}\label{a7}
	&\sum_{i=1}^{n}\mathrm{E}\left\{\Vert\xi_i\Vert^2I(\Vert\xi_i\Vert>\epsilon)\right\} \nonumber\\
	\le&\left(\frac{K}{r}\right)^{3/2}\frac{1}{\epsilon}\sum_{i=1}^{n}\mathrm{E}\left\{\Vert\xi_i\Vert^3\right\}\nonumber\\
	\le&\left(\frac{K}{r}\right)^{3/2}\frac{1}{\epsilon}\sum_{i=1}^{n}\mathrm{E}\left\{\frac{R^3_i\Vert \boldsymbol{B}^T_i\boldsymbol{\delta}\Vert^3\Vert\psi_\tau(u_i)\Vert^3}{n^3\pi^3_i}\right\}\nonumber\\
	=&\left(\frac{K}{r}\right)^{3/2}\frac{1}{\epsilon}\sum_{i=1}^{n}\frac{\mathrm{E}\left[R_i^3\right]\mid\boldsymbol{B}^T_i\boldsymbol{\delta}\mid^3\mathrm{E}\left\{\Vert\psi_\tau(u_i)\Vert^3\mid x_i(t)\right\}}{n^3\pi_i^3}\nonumber\\
	=& o_P(1),    	
\end{align}
where 
\begin{equation*}
	\mathrm{E}\left[R_i^3\right]=r(r-1)(r-2)\pi_i^3+3r(r-1)\pi_i^2+r\pi_i,
\end{equation*}
and the last equality holds by combining Assumption \ref{A6}, Lemma \ref{lem1} and the fact that $ \mid\psi_\tau(u_i)\mid\le 1 $. Thus, by Lindeberg-Feller central limit theorem, it can be concluded that as $ n \rightarrow\infty $, $ r \rightarrow\infty $, 
\begin{equation*}
	\frac{U_r-\mathrm{E}\left[U_r\right]}{\sqrt{\mathrm{Var}\left[U_r\right]}} \rightarrow N(0,1)
\end{equation*}
in distribution, which implies that the equation (\ref{a1}) holds because $ \mathrm{E}\left[U_r\right]=O(\sqrt{rK}K^{-(d+1)})=o_P(1) $. This completes the proof.
\end{proof}	
\begin{lemma}\label{lem4} 
	Let $ v_i=\sqrt{K/r}\boldsymbol{B}^T_i\mathrm{\delta} $. Under the same assumptions as Theorem \ref{Th3},
	\begin{eqnarray*}
		\sum_{i=1}^{n}\frac{R_i\int_{0}^{v_i}\{I(u_i \le s)-I(u_i\le 0)\}\mathrm{d}s}{n\pi_i}=\frac{K}{2}\boldsymbol{\delta}^T\boldsymbol{G}_{\tau}\boldsymbol{\delta} +o_P(1) .  	
	\end{eqnarray*} 	
\end{lemma}
\begin{proof}
Let
\begin{eqnarray*}
	M_r=\sum_{i=1}^{n}\frac{R_i\int_{0}^{v_i}\{I(u_i \le s)-I(u_i\le 0)\}\mathrm{d}s}{n\pi_i}.
\end{eqnarray*}
Since
\begin{align*}
	&\mathrm{E}\left\{\frac{R_i\int_{0}^{v_i}\left\{I(u_i \le s)-I(u_i\le 0)\right\}\mathrm{d}s}{n\pi_i}\right\}\\
	= &\mathrm{E}\left\{\mathrm{E}\left\{\frac{R_i\int_{0}^{v_i}\left\{I(u_i \le s)-I(u_i\le 0)\right\}\mathrm{d}s}{n\pi_i}\mid\mathcal{F}_n\right\}\right\}\nonumber    \\
	=& \frac{r}{n}\mathrm{E}\left\{\int_{0}^{v_i}\left\{I(u_i \le s)-I(u_i\le 0)\right\}\mathrm{d}s\mid x_i(t)\right\} \\
	=& \frac{r}{n}\int_{0}^{v_i}\left\{\mathrm{P}\left(y_i<\boldsymbol{B}^T_i\boldsymbol{\theta}_0+s\mid x_i(t)\right)-\mathrm{P}\left(y_i<\boldsymbol{B}^T_i\boldsymbol{\theta}_0\mid x_i(t)\right)\right\}\mathrm{d}s   \\
	=& \frac{\sqrt{rK}}{n} \int_{0}^{\boldsymbol{B}^T_i\boldsymbol{\delta}}\left\{\mathrm{P}\left(y_i<\boldsymbol{B}^T_i\boldsymbol{\theta}_0+l\sqrt{\frac{K}{r}}\mid x_i(t)\right)-\mathrm{P}\left(y_i<\boldsymbol{B}^T_i\boldsymbol{\theta}_0\mid x_i(t)\right)\right\}\mathrm{d}l    \\
	=& \frac{K}{n}\int_{0}^{\boldsymbol{B}^T_i\boldsymbol{\delta}}f_{\epsilon\mid \boldsymbol{X}(t)}(\boldsymbol{B}^T_i\boldsymbol{\theta}_0,x_i(t))l\mathrm{d}l\cdot(1+o_P(1))   \\
	=& \frac{K}{2n}f_{\epsilon\mid \boldsymbol{X}(t)}(\boldsymbol{B}^T_i\boldsymbol{\theta}_0,x_i(t))(\boldsymbol{B}^T_i\boldsymbol{\delta})^2(1+o_P(1)),
\end{align*}
we can obtain the total expectation of $ M_r $ as follows
\begin{eqnarray}\label{a8}
	\mathrm{E}\left[M_r\right]
	&=& \frac{K}{2n}\sum_{i=1}^{n}f_{\epsilon\mid \boldsymbol{X}(t)}(\boldsymbol{B}^T_i\boldsymbol{\theta}_0,x_i(t))(\boldsymbol{B}^T_i\boldsymbol{\delta})^2(1+o_P(1))   \nonumber\\
	&=& \frac{K}{2}\boldsymbol{\delta}^T\left(\frac{1}{n}\sum_{i=1}^{n}f_{\epsilon\mid\boldsymbol{X}(t)}(0+o(1),x_i(t))\boldsymbol{B}_i\boldsymbol{B}^T_i\right)\boldsymbol{\delta}(1+o_P(1)) \nonumber  \\
	&=& \frac{K}{2}\boldsymbol{\delta}^T\boldsymbol{G}_{\tau}\boldsymbol{\delta}(1+o_P(1)).
\end{eqnarray}
Now, we show the total variance of $ M_r $ satisfying $ \mathrm{Var}[M_r]=o_P(1) $. Note that the variance of $ M_r $ can be evaluated as 
\begin{eqnarray}\label{a9}
	\mathrm{Var}\left[M_r\right]
	&\le& \sum_{i=1}^{n}\mathrm{E}\left\{\frac{R_i\int_{0}^{v_i}\{I(u_i \le s)-I(u_i\le 0)\}\mathrm{d}s}{n\pi_i}\right\}^2 \nonumber    \\
	&\le&  \sqrt{\frac{K}{r}}\left\{\mathop{\rm{max}}\limits_{i=1,2,\dots,n}\frac{\Vert \boldsymbol{B}^T_i\boldsymbol{\delta}\Vert}{n\pi_i}\right\}\cdot \mathrm{E}\left[M_r\right]\nonumber\\
	&\le& \sqrt{\frac{K}{r}}\left\{\mathop{\rm{max}}\limits_{i=1,2,\dots,n}\frac{1}{n\pi_i}\right\}\cdot\left\{\mathop{\rm{max}}\limits_{i=1,2,\dots,n}\mid\boldsymbol{B}^T_i\boldsymbol{\delta}\mid\right\}\cdot \mathrm{E}\left[M_r\right],
\end{eqnarray}
where the second inequality 
is from the fact that 
\begin{eqnarray*}
	\int_{0}^{v_i}\{I(u_i \le s)-I(u_i\le 0)\}\mathrm{d}s
	&\le&\left\vert\int_{0}^{v_i}\left\vert\{I(u_i \le s)-I(u_i\le 0)\}\right\vert \mathrm{d}s\right\vert \nonumber\\
	&\le& \sqrt{\frac{K}{r}}\left\vert \boldsymbol{B}^T_i\boldsymbol{\delta}\right\vert,\quad i=1,2,\dots,n.
\end{eqnarray*}
Thus, from (\ref{a8}), (\ref{a9}) and Assumption \ref{A6}, and noting $ \mathrm{E}\left[M_r\right]=O(1) $, we have $ \mathrm{Var}\left[M_r\right]=o_P(\sqrt{K/r^3})=o_P(1) $. As a result, Lemma \ref{lem4} holds by Chebyshev's inequality.
\end{proof}

In the following, we present the proofs of Theorems \ref{Th1}, \ref{Th2}, \ref{Th3}, \ref{Th4}, and \ref{Th5} in turn.
\begin{proof}[Proof of Theorem~{\upshape\ref{Th1}} and {\upshape\ref{Th2}}] 
Theorem \ref{Th1} can be proved similar to Theorem 1 of  
\cite{Yoshida2013quantile}, and Theorem \ref{Th2} can be obtained directly from Theorem \ref{Th1} by considering Assumptions \ref{A4} and \ref{A5}. Here we omit the details.
\end{proof}

\begin{proof}[Proof of Theorem~{\upshape\ref{Th3}}] Let 				
\begin{eqnarray*}
	Z_r(\boldsymbol{\delta})&=&\sum_{i=1}^{n}\frac{R_i(\rho_\tau(u_i-v_i)-\rho_\tau(u_i))}{\pi_i}\nonumber\\
	&\;&+\frac{r\lambda}{2}(\boldsymbol{\theta}_0+\sqrt{\frac{K}{r}}\boldsymbol{\delta})^T \boldsymbol{D}_q(\boldsymbol{\theta}_0+\sqrt{\frac{K}{r}}\boldsymbol{\delta})-\frac{r\lambda}{2}\boldsymbol{\theta}_0^T\boldsymbol{D}_q\boldsymbol{\theta}_0,
\end{eqnarray*}
where $ u_i=y_i-\boldsymbol{B}^T_i\boldsymbol{\theta}_0 $ and $ v_i=\sqrt{r/K}\boldsymbol{B}^T_i\boldsymbol{\delta} $. It is easy to see that this function is convex and minimized at $\sqrt{r/K}(\boldsymbol{\tilde{\theta}}-\boldsymbol{\theta}_0)  $. 

On the other hand, using Knight's identity,
\begin{equation}\label{a10}
	\rho_\tau(u-v)-\rho_\tau(u)=-v\psi_\tau(u)+\int_{0}^{v}\{I(u\le s)-I(u\le 0)\}\mathrm{d}s,
\end{equation}
where $ \psi_\tau(u)=\tau-I(u<0) $, we have
\begin{equation}\label{a11}
	Z_r(\boldsymbol{\delta})=Z_{1r}(\boldsymbol{\delta})+Z_{2r}(\boldsymbol{\delta})+Z_{3r}(\boldsymbol{\delta})+Z_{4r}(\boldsymbol{\delta}),
\end{equation}
where 
\begin{align*}
	&Z_{1r}(\boldsymbol{\delta})=-\sqrt{\frac{K}{r}}\sum_{i=1}^{n}\frac{R_i}{\pi_i}\boldsymbol{B}^T_i\boldsymbol{\delta}\psi_\tau(u_i),\\
	&Z_{2r}(\boldsymbol{\delta})=\sum_{i=1}^{n}\frac{R_i\int_{0}^{v_i}\left\{I(u_i \le s)-I(u_i\le 0)\right\}\mathrm{d}s}{\pi_i},\\
	&Z_{3r}(\boldsymbol{\delta})=\frac{K\lambda}{2}\boldsymbol{\delta}^T \boldsymbol{D}_q\boldsymbol{\delta},\\
	&Z_{4r}(\boldsymbol{\delta})=\sqrt{rK}\lambda\boldsymbol{\theta}_0^T\boldsymbol{D}_q\boldsymbol{\delta}.       	
\end{align*}
From Lemma \ref{lem3}, $ Z_{1r}(\boldsymbol{\delta}) $ in (\ref{a11}) satisfies
\begin{equation}\label{a12}
	\frac{Z_{1r}(\boldsymbol{\delta})}{n}=-\sqrt{K}\boldsymbol{W}^T\boldsymbol{\delta}+o_P(1),
\end{equation}
where $ \left\{\tau(1-\tau)(\boldsymbol{V}_{\pi}+\eta \boldsymbol{G})\right\}^{-1/2}\boldsymbol{W}\rightarrow N(0,1) $ in distribution. Furthermore, Lemma \ref{lem4} and $ Z_{3r}(\boldsymbol{\delta}) $ in (\ref{a11}) yield
\begin{equation}\label{a13}
	\frac{Z_{2r}(\boldsymbol{\delta})}{n}+\frac{Z_{3r}(\boldsymbol{\delta})}{n}=\frac{K}{2}\boldsymbol{\delta}^T\left(\boldsymbol{G}_{\tau}+\frac{\lambda}{n}\boldsymbol{D}_q\right)\boldsymbol{\delta} +o_P(1)
	=\frac{K}{2}\boldsymbol{\delta}^T\boldsymbol{H}_{\tau}\boldsymbol{\delta} +o_P(1).   	
\end{equation}
Therefore, from (\ref{a11}),(\ref{a12}) and (\ref{a13}), we can obtain
\begin{equation*}
	\frac{Z_{r}(\boldsymbol{\delta})}{n}=-\sqrt{K}\boldsymbol{W}^T\boldsymbol{\delta}+\frac{K}{2}\boldsymbol{\delta}^T\boldsymbol{H}_{\tau}\boldsymbol{\delta} +\frac{\sqrt{rK}}{n}\lambda\boldsymbol{\theta}_0^T \boldsymbol{D}_q\boldsymbol{\delta}+ o_P(1).   	
\end{equation*}    

Since $ Z_{r}(\boldsymbol{\delta})/n$ is convex with respect to $ \boldsymbol{\delta} $ and has unique minimizer, from the corollary in page 2 of 
\cite{Hjort2011}, its minimizer, $ \sqrt{r/K}(\boldsymbol{\tilde{\theta}}-\boldsymbol{\theta}_0)$, satisfies that
\begin{equation*}
	\sqrt{\frac{r}{K}}(\boldsymbol{\tilde{\theta}}-\boldsymbol{\theta}_0)=\boldsymbol{H}_{\tau}^{-1}\left(\frac{1}{\sqrt{K}}\boldsymbol{W} -\sqrt{\frac{r}{K}}\cdot\frac{\lambda}{n}\boldsymbol{D}_q\boldsymbol{\theta}_0\right)+ o_P(1).   	
\end{equation*}
Because the random vector is only $ \boldsymbol{W}$ in asymptotic form of $ \boldsymbol{\tilde{\theta}} $ and $ \tilde{\beta}(t)-\beta_0(t)=\boldsymbol{B}^{T}(t)(\boldsymbol{\tilde{\theta}}-\boldsymbol{\theta}_0) $, the expectation of $ \tilde{\beta}(t)-\beta_0(t) $ can be written as
\begin{equation*}
	\mathrm{E}\{\tilde{\beta}(t)-\beta_0(t)\}=b_{\lambda}(t)(1+o_P(1)),   	
\end{equation*}
where $ b_{\lambda}(t)=-\frac{\lambda}{n}\boldsymbol{B}^{T}(t)\boldsymbol{H}_{\tau}^{-1}\boldsymbol{D}_q\boldsymbol{\theta}_0 $. Together with $ \tilde{\beta}(t)-\beta(t)= \tilde{\beta}(t)-\beta_0(t)+ \beta_0(t)-\beta(t)  $, we have the asymptotic bias of $ \tilde{\beta}(t)-\beta(t) $ as
\begin{eqnarray*}
	\mathrm{E}\{\tilde{\beta}(t)-\beta_0(t)\}=b_a(t)(1+o_P(1))+b_{\lambda}(t)(1+o_P(1)).
\end{eqnarray*}
Thus, we have
\begin{align*}
	&\{\boldsymbol{B}(t)^T\boldsymbol{V}\boldsymbol{B}(t)\}^{-1/2}\sqrt{r/K}(\tilde{\beta}(t)-\beta(t)-b_a(t)-b_{\lambda}(t))\\
	=&\{\boldsymbol{B}(t)^T\boldsymbol{V}\boldsymbol{B}(t)\}^{-1/2}\boldsymbol{B}^T(t)\boldsymbol{H}_{\tau}^{-1}\frac{1}{\sqrt{K}}\boldsymbol{W}+o_P(1).
\end{align*}
Combining the fact that 
\begin{align*} \{\boldsymbol{B}(t)^T\boldsymbol{V}\boldsymbol{B}(t)\}^{-1/2}\boldsymbol{B}^T(t)\boldsymbol{V}\boldsymbol{B}(t)\{\boldsymbol{B}(t)^T\boldsymbol{V}\boldsymbol{B}(t)\}^{-1/2}=1, \end{align*} 
by the definition of $ \boldsymbol{W} $ and Slutsky's Theorem, we can obtain for $ t\in [0,1] $, as $r, n\rightarrow \infty $,
\begin{equation*}
	\{\boldsymbol{B}(t)^T\boldsymbol{V}\boldsymbol{B}(t)\}^{-1/2}\sqrt{r/K}(\tilde{\beta}(t)-\beta(t)-b_a(t)-b_{\lambda}(t))\rightarrow N(0,1).
\end{equation*}
Further, from the discussions before Theorem \ref{Th2}, we know that $ b_{\lambda}(t) $ and $ b_a(t) =o_P(1)$ are negligible. Thus, we have
\begin{equation*}
	\{\boldsymbol{B}(t)^T\boldsymbol{V}\boldsymbol{B}(t)\}^{-1/2}\sqrt{r/K}(\tilde{\beta}(t)-\beta(t))\rightarrow N(0,1).
\end{equation*}
So Theorem \ref{Th3} is proved.
\end{proof}

\begin{proof}[Proof of Theorem~{\upshape\ref{Th4}}]
Note that
\begin{eqnarray*}
	\mathrm{tr}(\boldsymbol{V})
	&=&\frac{\tau(1-\tau)}{K}\mathrm{tr}\left[\boldsymbol{H}^{-1}_{\tau}\left(\sum_{i=1}^{n}\frac{\boldsymbol{B}_i\boldsymbol{B}^T_i}{n^2\pi_i}+\eta \sum_{i=1}^{n}\frac{\boldsymbol{B}_i\boldsymbol{B}^T_i}{n}\right)\boldsymbol{H}^{-1}_{\tau}\right]\\
	&=&\frac{\tau(1-\tau)}{Kn^2}\sum_{i=1}^{n}\mathrm{tr}\left[\frac{\boldsymbol{H}^{-1}_{\tau}\boldsymbol{B}_i\boldsymbol{B}^T_i\boldsymbol{H}^{-1}_{\tau}}{\pi_i}\right]\\
	&\;&+\frac{\tau(1-\tau)\eta}{Kn}\sum_{i=1}^{n}\mathrm{tr}\left[ H^{-1}_{\tau}\boldsymbol{B}_i\boldsymbol{B}^T_i\boldsymbol{H}^{-1}_{\tau}\right]\\
	&=&\frac{\tau(1-\tau)}{Kn^2}\sum_{i=1}^{n}\frac{\Vert H^{-1}_{\tau}\boldsymbol{B}_i\Vert_2^2}{\pi_i}+\frac{\tau(1-\tau)\eta}{Kn}\sum_{i=1}^{n}\Vert\boldsymbol{H}^{-1}_{\tau}\boldsymbol{B}_i\Vert_2^2\\
	&=&\frac{\tau(1-\tau)}{Kn^2}\left(\sum_{i=1}^{n}\pi_i\right)\left(\sum_{i=1}^{n}\frac{\Vert H^{-1}_{\tau}\boldsymbol{B}_i\Vert_2^2}{\pi_i}\right)+\frac{\tau(1-\tau)\eta}{Kn}\sum_{i=1}^{n}\Vert H^{-1}_{\tau}\boldsymbol{B}_i\Vert_2^2\\
	&\ge& \frac{\tau(1-\tau)}{Kn^2}\left(\sum_{i=1}^{n}\Vert H^{-1}_{\tau}\boldsymbol{B}_i\Vert_2\right)^2+\frac{\tau(1-\tau)\eta}{Kn}\sum_{i=1}^{n}\Vert H^{-1}_{\tau}\boldsymbol{B}_i\Vert_2^2,
\end{eqnarray*}
where the last inequality is from the Cauchy-Schwarz inequality and the equality in it holds if and only if when $ \pi_i \propto \Vert \boldsymbol{H}^{-1}_{\tau}\boldsymbol{B}_i\Vert_2 $. So the proof is completed by considering $\sum_{i=1}^{n}\pi_i=1 $.
\end{proof}

\begin{proof}[Proof of Theorem~{\upshape\ref{Th5}}] 
	Note that
\begin{eqnarray*}
	\mathrm{tr}\left[\boldsymbol{V}_{\pi}\right]
	&=&\mathrm{tr}\left(\sum_{i=1}^{n}\frac{\boldsymbol{B}_i\boldsymbol{B}^T_i}{n^2\pi_i}\right)=\frac{1}{n^2}\sum_{i=1}^{n}\mathrm{tr}\left(\frac{\boldsymbol{B}_i\boldsymbol{B}^T_i}{\pi_i}\right)\\
	&=&\frac{1}{n^2}\sum_{i=1}^{n}\frac{\Vert \boldsymbol{B}_i\Vert_2^2}{\pi_i}=\frac{1}{n^2}\left(\sum_{i=1}^{n}\pi_i\right)\left(\sum_{i=1}^{n}\frac{\Vert\boldsymbol{B}_i\Vert_2^2}{\pi_i}\right)\\
	&\ge& \frac{1}{n^2}\left(\sum_{i=1}^{n}\Vert\boldsymbol{B}_i\Vert_2\right)^2,
\end{eqnarray*}
where the last inequality is from the Cauchy-Schwarz inequality and the equality in it holds if and only if when $ \pi_i \propto \Vert \boldsymbol{B}_i\Vert_2 $. So the proof is completed by considering $\sum_{i=1}^{n}\pi_i=1 $.
\end{proof}

\end{appendices}



\end{document}